\documentclass[reqno,12pt]{amsart}

\usepackage[T1]{fontenc}
\usepackage{amssymb}
\usepackage{amssymb, latexsym}
\usepackage{hyperref}
\usepackage{amsmath}
\usepackage{amscd}
\usepackage{enumerate}
\usepackage{amsfonts}
\usepackage{graphicx}
\usepackage[all,cmtip]{xy}
\usepackage{mathrsfs}
\usepackage{slashed}
\usepackage{xcolor}
\usepackage[numbers,sort&compress]{natbib}

\def\ge{\geqslant}
\def\le{\leqslant}
\def\geq{\geqslant}
\def\leq{\leqslant}

\allowdisplaybreaks

\renewcommand{\Re}{\mathop{\mathrm{Re}}}
\renewcommand{\Im}{\mathop{\mathrm{Im}}}

\newcommand{\C}{{\mathbb{C}}}
\newcommand{\R}{{\mathbb{R}}}

\newcommand{\eps}{{\varepsilon}}
\newcommand{\ph}{{\varphi}}

\newcommand{\HHH}{\mathcal{H}}
\newcommand{\KKK}{\mathcal{K}}

\makeatletter

\newcommand{\Rmnum}[1]{\expandafter\@slowromancap\romannumeral #1@}
\makeatother

\theoremstyle{plain}
\newtheorem{theorem}{Theorem}
\newtheorem{proposition}[theorem]{Proposition}
\newtheorem{property}[theorem]{Property}
\newtheorem{lemma}[theorem]{Lemma}

\theoremstyle{definition}

\newtheorem{remark}[theorem]{Remark}

\numberwithin{equation}{section}
\numberwithin{theorem}{section}

\numberwithin{equation}{section}

\allowdisplaybreaks
\begin{document}


\title[Blow up versus scattering for the focusing NLH with potential]{Blow up versus scattering below the mass-energy threshold for the focusing NLH with potential}

\author[S. Ji]{Shuang Ji}
\address{College of Science, China Agricultural University, \ Beijing, \ China, \ 100193,}
\email{jishuang@cau.edu.cn}

\author[J. Lu]{Jing Lu}
\address{College of Science, China Agricultural University, \ Beijing, \ China, \ 100193, }
\email{lujing@cau.edu.cn}


\date{\today}


\begin{abstract}\noindent
In this paper, we study the blow up and scattering result of the solution to the focusing nonlinear Hartree equation with potential
$$i\partial_t u +\Delta u - Vu = - (|\cdot|^{-3} \ast |u|^2)u, \qquad (t, x) \in \R \times \R^5 $$
in the energy space ${H}^1(\R^5)$ below the mass-energy threshold. 
The potential $V$ we considered is an extension of Kato potential in some sense.  
We extend the results of Meng  \cite{meng} to nonlinear Hartree equation with potential $V$ under some conditions. 
By establishing a Virial-Morawetz estimate and a scattering criteria, we obtain the  scattering theory based on the method from Dodson-Murphy \cite{dodson}. 
\end{abstract}

\maketitle

\section{Introduction}
In this paper, we consider the Cauchy problem of the following nonlinear Hartree equation with potential:
 \begin{equation}\label{NLHv}\tag{$\text{NLH}_{\text{V}}$}
   \left\{ \aligned
   & i\partial_t u +\Delta u - Vu = \mu (|\cdot|^{-\gamma} \ast |u|^2)u, \\
   & u(0)=u_0\in {H}^1(\R^d),
   \endaligned
  \right. \qquad (t, x) \in \R \times \R^d
\end{equation}
where $ u: \R \times \R^d \rightarrow \C $ is the wave function, $ V : \R^d \to \R $ is a  potential, $ \mu \ne 0 $, $ 0 < \gamma < d $, and $ \ast $ denotes the convolution of spacial variable. 
\eqref{NLHv} is divided into two cases according to the sign of $ \mu $, which is called focusing if $ \mu < 0 $ and defocusing if $ \mu > 0 $.

The solution to \eqref{NLHv} satisfies the laws of mass conservation and energy conservation, which can be expressed respectively by 
\begin{equation}\label{mass}
	M(u) = \int_{\R^d} \vert u(t,x) \vert ^2 dx = M(u_0),
\end{equation}
\begin{equation}\label{energy}
	E(u) = \frac{1}{2} \int_{\R^d} \big( \vert \nabla u \vert ^2 + V |u|^2 \big) dx + \frac{\mu}{\gamma+1}\int_{\R^d}\int_{\R^d} \frac{|u(x)|^2 |u(y)|^2}{|x-y|^{\gamma}}dxdy = E(u_0).
\end{equation}

Before showing our main theorem, we first recall some known results for the nonlinear Hartree equation without potential, that is
 \begin{equation}\label{NLH}\tag{$\text{NLH}$}
	\left\{ \aligned
	& i\partial_t u +\Delta u = \mu (|\cdot|^{-\gamma} \ast |u|^2)u,  \\
	& u(0)=u_0\in {H}^1(\R^d).
	\endaligned
	\right. \qquad (t,x)\in \R\times\R^d,
\end{equation}
\eqref{NLH} enjoys the scaling symmetry
\begin{equation}\label{scaling}
	u(t,x) \mapsto u^\lambda(t,x) : = \lambda^{\frac{d+2-\gamma}{2}} u(\lambda^2t, \lambda x).
\end{equation}
This symmetry identifies $\dot H_x^{s_{c}}(\R^d)$ as the scaling-critical space of initial data, where $s_{c}=\frac{\gamma}{2}-1$. When  $s_c<0$ (or $\gamma<2$), we call \eqref{NLH} as \emph{mass-subcritical} problem. When  $s_c=0$ (or $\gamma=2$), we call \eqref{NLH} as \emph{mass-critical} problem. When  $0 < s_c < 1$ (or  $2<\gamma<4$), we call \eqref{NLH} as  \emph{energy-subcritical} problem. When  $s_c=1$ (or $\gamma=4$), we call \eqref{NLH} as \emph{energy-critical} problem.

Formally, the solution of \eqref{NLH} satisfies the conservation of mass and energy:
\begin{equation*}\label{mass1}
	M_0(u) = \int_{\R^d} \vert u(t,x) \vert ^2 dx = M_0(u_0),
\end{equation*}
\begin{equation*}\label{energy1}
	E_0(u) = \frac{1}{2} \int_{\R^d} \vert \nabla u \vert ^2 dx + \frac{\mu}{\gamma+1}\int_{\R^d}\int_{\R^d} \frac{|u(x)|^2 |u(y)|^2}{|x-y|^{\gamma}}dxdy = E_0(u_0).
\end{equation*}

The Hartree equation is an important dispersive equation with famous scientific 
research background. Its study  arises in  Boson stars theory, and  can be described as  a continuous-limit model for mesoscopic molecular structures    in Chemistry. There has been a dramatic increase in the research of nonlinear Hartree equation \eqref{NLH}. 
The well-posedness, ill-posedness and the small scattering result  of \eqref{NLH} in $H^1(\R^d)$ have been studied by Miao-Xu-Zhao in \cite{cauchy} with $0<\gamma<d$. 
\begin{itemize}
    \item For the defocusing ones: In the energy-subcritical case, Miao-Xu-Zhao obtained  the scattering theory by I-method in \cite{energysubcritical}. 
    Later, Miao-Xu-Zhao in \cite{gwpandscatter} obtained the global well-posedness and scattering for energy-critical Hartree equation using the concentration compactness argument from \cite{KM}. 
    For the energy-critical case, Li-Miao-Zhang \cite{limiaozhang} established the small data scattering result for \eqref{NLH} in $d\ge5$. 
    Later, Miao-Xu-Zhao \cite{mxzblowup} got the scattering and blow up results for spherically symmetric initial data. 
    \item For the focusing ones: For the mass-critical case, Miao-Xu-Zhao \cite{masscritical} drew our attention to the scattering result for \eqref{NLH} in the radial solution under the condition that the initial data are strictly less than that of ground state. 
    Miao-Wu-Xu \cite{miaowuxu} studied the dynamics of the radial solutions with the threshold energy by the moving plane method in the global form of a positive solution. 
    In particular, Gao-Wu in \cite{gao} got the scattering result  and the blow up result for the $\dot{H}^{\frac{1}{2}}$-critical \eqref{NLH} with initial data in energy space $H_x^1$. 
    Later, using the method from Dodson-Murphy \cite{dodson}, Meng in \cite{meng} gave a new proof of the scattering below the ground state that avoids the use of concentration compactness in \cite{gao}. 
\end{itemize}

We summarize their result as follows: 
\begin{theorem}[\cite{meng}]\label{Meng-Gao-Wu}
    For $ (\mu, \gamma, d) = (-1, 3, 5) $ and $ V \equiv 0 $ in \eqref{NLHv}, if  $u_0 \in H^1(\R^5)$ is radially symmetric satisfying
    \begin{equation*}\label{ground}
		E(u_0)M(u_0)<E(Q)M(Q), \quad\quad || u_0||_{\dot{H}^1}||u_0||_{L^2}<||\nabla Q||_{L^2}||Q||_{L^2},
    \end{equation*}
    then the global solution scatters in $H^1$  in both time directions, where $Q$ is the unique radial Schwartz solution to
    \begin{equation}\label{elliptic}
    -\Delta Q+Q=(|\cdot|^{-3}\ast|Q|^2)Q.
    \end{equation}
\end{theorem}
%

In this paper, we assume the potential $ V $ in \eqref{NLHv} satisfies the following hypothesis. 
\begin{enumerate}
\item[(\text{\textbf{H1}})]\label{H1}  
 $ V : \R^d \to \R $ is real-valued and $V \in L^{\frac{d}2}( \R^d )$,
\item[(\text{\textbf{H2}})]\label{H2} the operator $ \HHH := -\Delta + V $ has no eigenvalues. 
\end{enumerate}
\begin{remark} A natural question is that why the potential $V$ should satisfy $\mathbf{(H1)}$ and $\mathbf{(H2)}$?    
    \begin{itemize}
        \item A typical potential satisfying $\mathbf{(H1)}$ and $\mathbf{(H2)}$ is Kato potential under some conditions. 
        It follows from \cite{Hong} that if the potential $V$ satisfies 
        \begin{equation}\label{v2}
        V \in \KKK_0 \cap L^{\frac{d}2}(\R^d), \quad d=3
        \end{equation} 
        with the Kato norm
$
\Vert V \Vert_{\KKK_0} := \sup_{x \in \R^d} \int_{\R^d} |V(y)| |x-y|^{2-d} dy 
$ 
and 
\begin{equation}\label{v1}
\Vert V_- \Vert_{\KKK_0} < 4 \pi
\end{equation}
for $ V_- := \min \{ V, 0 \} $, then the Schr\"odinger operator $ \HHH$ is positive definite when the negative part of a potential is small (It has been proved in Lemma A.1 \cite{Hong}). Since  $ \HHH$  is positive, and thus it has no negative eigenvalue. 
Besides, there is no positive eigenvalue or resonance by the theory of \cite{IJ}. 
Therefore, in this paper we just need $\mathbf{(H2)}$ instead of the conditions $V \in \KKK_0$ and $\Vert V_- \Vert_{\KKK_0} < 4 \pi$. In some sense, the potential we considered is an extension of Kato potential.
   \item Because of the absence of the potential $V$, then the standard Sobolev norms and the Sobolev norms associated with $ \HHH$ are different. 
   So we should find the relations of them. 
   Thus it is necessary to establish the norm equivalence. As seen in $\mathbf{(H1)}$, we need $V \in L^{\frac{d}2}( \R^d )$. 
   This condition is just used to prove the Sobolev inequalities and equivalence of Sobolev spaces, which will be used in the whole paper(see Lemma \ref{si}, \ref{equivalence} for details).
\end{itemize}
\end{remark} 

 In recent decades, there are also some work to study the nonlinear Schr\"odinger equations with potential $V$ satisfying \eqref{v2} and \eqref{v1}, 
\begin{equation}\label{NLSv}\tag{$ \text{NLS}_{\text{V}} $}
   \left\{ \aligned
   & i\partial_t u +\Delta u - Vu = \mu |u|^{\alpha}u, \\
   & u(0)=u_0\in {H}^1(\R^3),
   \endaligned
  \right. \qquad (t, x) \in \R \times \R^3
\end{equation} 
The local well-posedness result and the energy scattering of \eqref{NLSv}  in the  defocusing case were proved by \cite{Hong}.  
For the focusing case, the energy scattering of \eqref{NLSv} below the threshold was first proved by \cite{Hong} with $\alpha=2$ using the concentration compactness argument of Kenig-Merle in \cite{KM}. 
Later, using the method form Murphy-Dodson \cite{dodson}, \cite{masaru1} extended the results in \cite{Hong} to the intercritical case $\frac43<\alpha<4$ with radial symmetric initial data. 
Besides, they prove the blow up criteria for the equation by applying the argument of Du-Wu-Zhang in \cite{DWZ}. For the other kinds of potentials, such as the inverse square potential and Coulomb potential, see \cite{AA,ben,boni,BPS,campos,chen,dynamics,GWY,masaru,KMVZZ,KMVZZ1,KMVZ,orbital,cj,coulomb,mizutani,ZZ,Z} for details.

In this paper, we study the scattering of \eqref{NLHv} and extend Meng's result of focusing ($ \mu = -1 $) energy-subcritical ($ s_c = \frac12 $) Hartree equation in 5D in \cite{meng} to the one with potential $ V $. 
Besides, we consider the blow up theory about the equation \eqref{NLHv}. For convenience, we define the Hilbert space with the inner product
\begin{equation}
	\|\nabla_V  f\|_{L_x^2}^2:=<\HHH f,f>=\int | \nabla f|^2dx+\int V|f|^2dx,\quad f\in H_x^1(\R^d).
\end{equation}
Then our main results are as follows. 
\begin{theorem}\label{conclusion}
For $ (\mu, \gamma, d) = (-1, 3, 5) $ in \eqref{NLHv}, we assume $ V \ge 0 $ and satisfies $\mathbf{(H1)}$ and $\mathbf{(H2)}$. 
Let $ u_0 \in H^1 $  and satisfies
\begin{equation}\label{u0q}
E(u_0)M(u_0) < E_0(Q)M_0(Q),
\end{equation}	
where $ Q $ is the solution to \eqref{elliptic}.
  \begin{enumerate}[\rm (i)]
      \item {\rm (Global existence and Scattering)} If $ u_0$ is radial and
	\begin{equation}\label{lambda0q}
		||\nabla_V u_0||_{L^2}||u_0||_{L^2}<||\nabla Q||_{L^2}||Q||_{L^2},
	\end{equation}
	then  the solution u to \eqref{NLHv} exists globally in time and satisfies
	\begin{equation}\label{lambdatq}
		||\nabla_V u(t)||_{L^2}||u(t)||_{L^2}<||\nabla Q||_{L^2}||Q||_{L^2},
	\end{equation}
	for all $t \in \R$. 
        Moreover, if $\ x \cdot \nabla V \leq 0 $ and $ x \cdot \nabla V \in L^{\frac52} $, then the global solution scatters in $H^1$  in both time directions, that is, there exist $u_\pm \in H^1$ such that
	\begin{equation}\label{inf}
		\lim_{t \to \pm\infty }||u(t)-e^{-it\HHH}u_\pm||_{H^1}=0.
	\end{equation}
	\item {\rm (Blow up)} If
	\begin{equation}\label{lambda0>q}
		||\nabla_V u_0||_{L^2}||u_0||_{L^2}>||\nabla Q||_{L^2}||Q||_{L^2},
	\end{equation}
	then the solution u to \eqref{NLHv}  satisfies
	\begin{equation}\label{lambdat>q}
		||\nabla_V u(t)||_{L^2}||u(t)||_{L^2}>||\nabla Q||_{L^2}||Q||_{L^2},
	\end{equation}
	for all $t \in (-T_*,T^*)$, where $(-T_*,T^*)$ is the maximal time interval of existence. Moreover, if $\ x \cdot  \nabla V \in L^{\frac{5}{2}},\  2V+ x\cdot \nabla V \geq 0$, then either $T^*<+\infty$ and 
    $$\lim_{t \to T^*}\Vert \nabla u(t)\Vert_{L^2}=\infty,$$  
 or $T^*=+\infty$ and there exists a time sequence $t_n \to +\infty$ such that $$\lim_{n \to +\infty } \Vert \nabla u(t_n)\Vert_{L^2}=\infty.$$ A similar conclusion holds for $T_*$.
  \end{enumerate}
\end{theorem}
\begin{remark}Compared with \cite{masaru1}, our difficulty lies in the fact that Hartree
equation \eqref{NLHv} has a nonlocal term. And we extend the result in Meng \cite{meng} to the nonlinear Hartree equation with a potential. A natural question is that the dynamics of the solution to the equation \eqref{NLHv}
when beyond the mass-energy threshold. We will answer this problem in the later paper. \end{remark}
%

$\mathit{\textbf{The~sketch~of~scattering:}}$  In order to prove the scattering theory, we first prove a scattering criteria (see Section 3 for details). That is if $u \in H^1$ is a solution to  \eqref{NLHv} with $(\mu,\gamma,d) =(-1,3,5)$ satisfying
	$\sup_t \Vert u(t) \Vert_{H^1} :=\mathcal{E} < +\infty,$
	then there exist two constants $R>0$  and $\eps > 0$, depending only on $\mathcal{E}$, such that if
	\begin{equation}\label{scatter-1}
		\lim_{t \to \infty} \int_{B(0,R)}|u(t,x)|^2dx \leq \eps^{1+},
	\end{equation}
	then $u$ scatters forward in time in $H^1(\R^5)$.  Thus, our aim is to prove \eqref{scatter-1} holds. Indeed, for a fixed $0<R_0\ll R$, we have
\begin{align*}
	\Vert \chi_{R_0}u(t)\Vert_{L_x^2(\R^5)}^4 &= \int_{\R^5}|\chi_{R_0}u(t,x)|^2dx\int_{\R^5}|\chi_{R_0}u(t,y)|^2dy\\ \notag
	&\leq 8 R_0^3\int_{|x|\leq R_0}\int_{|y|\leq R_0}\frac{|u(t,x)|^2|u(t,y)|^2}{|x-y|^3}dxdy\\ \notag
	&\leq 8 R_0^3\int_{|x|\leq R}\int_{|y|\leq R}\frac{|u(t,x)|^2|u(t,y)|^2}{|x-y|^3}dxdy\\ \notag
	&=8R_0^3P(\chi_Ru).
\end{align*}
Making use of the Morawetz estimate (see Proposition \ref{morawetz} for details), we deduce
\begin{align*}
	\frac{1}{T}\int_{0}^{T}\Vert \chi_{R_0}u(t)\Vert_{L_x^2(\R^5)}^4 dt& \lesssim \frac{R_0^3}{T}\int_{0}^{T}P(\chi_Ru)dt\\ \notag
	&\lesssim_{\delta,u} \frac{R_0^3}{T}+\frac{R_0^3}{R}+o_R(1) \to 0,\  when \ T,R \to \infty,
\end{align*}
which means \eqref{scatter-1} holds. So we derive that $u$ scatters in $H_x^1$.

$\mathit{\textbf{The~sketch~of~blow~up:}}$
 If $u_0 \in L^2(|x|^2dx)$, then the solution will blow up in finite time by the classical method from Glassey in \cite{glassey}. If $u_0 \in H_x^1$, then we make use of the blow up criteria (see Proposition \ref{blowupcriteria} for details), we deduce that $u$ will blow up in finite or infinite time in the sense of $$\lim_{n \to +\infty } \Vert \nabla u(t_n)\Vert_{L^2}=\infty.$$ 

\begin{remark} 
    The condition that the initial data  $u_0$ is radial is applied to the proof of Proposition \ref{scattercriterion} (scattering criteria) and Proposition \ref{morawetz} (Morawetz estimate). 
    The positivity $V \geq 0$ is used to prove Property \ref{coer} (coercivity), Proposition \ref{morawetz} (Morawetz estimate) and Proposition \ref{blowupcriteria}  (blow up criteria).
    The condition $x \cdot \nabla V \leq 0$  is used to prove Proposition \ref{morawetz} (Morawetz estimate) to obtain the boundedness of the Strichartz norm. 
    The condition $2V + x \cdot \nabla V \geq 0$ is employed to establish Lemma \ref{kulemma} and Proposition \ref{blowupcriteria} (blow up criteria).  
    The condition $x \cdot \nabla V \in L^{\frac{5}{2}}$ is significant to establish the key results related to Proposition \ref{morawetz} (Morawetz estimate) and Proposition \ref{blowupcriteria} (blow up criteria).
\end{remark}

\textbf{Outline of the paper:} In Section 2, we introduce some useful inequalities, including Sobolev inequalities and equivalence  of Sobolev spaces. In Section 3, we concentrate on local well-posedness and the properties ground state. In Section 4, we give the proof of scattering by establishing scattering criteria and Morawetz estimate. In Section 5,  we consider 
 the blow up criteria in order to prove blow up result.
 
\section{Preliminaries}
Firstly, we introduce the notation and several fundamental lemmas needed in this paper.
The notation $A\lesssim B$ means that
${A}\leqslant{CB}$ for some constant $C>0$. Likewise, if ${A}\lesssim{B}\lesssim{A}$,
we say that ${A}\thicksim{B}$. We use $L^{r}_{x}(\mathbb{R}^d)$ to denote the Lebesgue
space of functions $f:\mathbb{R}^{d}\rightarrow{\mathbb{C}}$ with norm
$$\|f\|_{L^{r}}:=\Big(\int_{\mathbb{R}^d}|f(x)|^{r}dx\Big)^{\frac{1}{r}}$$
finite, with the usual modifications when $r=\infty$.
We also use the space-time Lebesgue spaces $L^{q}_{t}L^{r}_{x}$ which are equipped with the norm
$$\|f\|_{L^{q}_{t}L^{r}_{x}}:=\Big(\int_{I}\|f\|^{q}_{L^r_{x}}dt\Big)^{\frac{1}{q}}$$
for any space-time slab $I\times{\mathbb{R}^d}$.

\subsection{Some useful inequalities}
In this subsection, we will show some fundamental inequalities, which will be essential for the following sections.
\begin{lemma}[H\"older's inequality, \cite{holder}]\label{holder}
	Let $\Omega \subset \R^d$ is an open set, $f\in L^p(\Omega)$, $g \in L^{p'}(\Omega)$, $1\leq p \leq +\infty$. Then $f \cdot g \in L^1(\Omega)$, and
	\[\|fg\|_{L^1(\Omega)} \leq \|f\|_{L^p(\Omega)}\|g\|_{L^{p'}(\Omega)} ,\]
	where $(p,p')$ are conjugate pairs i.e. $\frac{1}{p}+\frac{1}{p'}=1$.
\end{lemma}

\begin{lemma}[Hardy-Littlewood-Sobolev inequality, \cite{hls}]\label{hls}
	If $1 < p < q < +\infty$, $0 < \gamma < d$ and
	\[\frac{1}{q}=\frac{1}{p}+\frac{\gamma}{d}-1,\]
	then
	\[\||\cdot|^{-\gamma} \ast f\|_{L^q(\R^d)} \lesssim \|f\|_{L^p(\R^d)}.\]
\end{lemma}

\begin{lemma}[Gagliardo-Nirenberg inequality, \cite{gn}]\label{gn}
	Let $1 \leq q,r \leq +\infty$, $0 < \sigma <s <+\infty$, then
	\[\|f\|_{\dot{W}^{\sigma,p}} \lesssim \|f\|_{L^q}^{\theta}\|f\|_{\dot{W}^{s,r}}^{1-\theta},\]
	where $\frac{1}{p}=\frac{\theta}{q}+\frac{1-\theta}{r}$ and $\sigma=s(1-\theta).$
\end{lemma}

\begin{lemma}[Interpolation inequality]\label{interpolation}
Let $1 < p,p_0,p_1 < +\infty$ and $\theta \in [0,1]$, then it holds that 
\[
\Vert u \Vert_{L^p} \lesssim \Vert u \Vert_{L^{p_0}}^{1-\theta} \Vert u \Vert_{L^{p_1}}^\theta
\]
where $\frac{1}{p}=\frac{1-\theta}{p_0}+ \frac{\theta}{p_1}$.
\end{lemma}
\begin{proof}
Refer to Lemma 2.3 in \cite{interpolation}, we let $s=s_1=0$. Thus we obtain the inequality.
\end{proof}


\begin{lemma}[Radial Sobolev embedding, \cite{gao}]\label{rse}
	Let $f\in H_x^1(\R^d)$ be radically symmetric. Then
	\[\||x|^{\frac{d-1}{2}}f\|^2_{L_x^{\infty}}\lesssim \|f\|_{H_x^1} \| \nabla f \|_{L^2}.\]
\end{lemma}

\begin{lemma}[Continuity method, \cite{meng}]\label{continuity}
	Let $a,b>0$, $p>1$, $b<\frac{a}{(2a)^p}$. Then it holds that
	\[
	f(x)=a-x+bx^p<0, \quad for \ some\  x>0.
	\]
	Suppose $I \subset \R$ is a interval, $\psi(t) \in C(I;\R^+)$ satisfying
	\[
	\psi(t) \leq a+b\psi^p(t), \quad \forall \ t \in I.
	\]
	If $\psi(t_0)\leq x_0$, $t_0 \in I$, then
	\[
	\psi(t) \leq x_0, \quad \forall \  t \in I,
	\]
	where $x_0$ is the first positive zero of $f(x)$.
\end{lemma}

\subsection{Harmonic analysis of operator $ \HHH $}
The Fourier transform on $\R^d$ is defined by
\begin{align*}
	\hat{f}(\xi):=(2\pi)^{-\frac{d}{2}}\int_{\R^d}e^{-ix\cdot\xi}f(x)\,dx,
\end{align*}
giving rise to the fractional differentiation operator $|\nabla|^s$, defined by
\begin{align*}
	|\nabla|^sf(x):=\mathcal{F}^{-1}_{\xi}(|\xi|^s\hat{f}(\xi))(x).
\end{align*}
Thus define the standard homogeneous Sobolev space norms:
\begin{align*}
	\|f\|_{\dot{W}_x^{s,p}(\R^d)}:=\||\nabla|^sf\|_{L_x^p(\R^d)}.
\end{align*}
We define the norms of the homogeneous and inhomogeneous Sobolev spaces associated to $\HHH := -\Delta +V$ as the closure of $C_0^{\infty}(\R^d)$
\[
\|f\|_{\dot{W}_V^{s,r}}:=\|\nabla_V^{s}f \|_{L^r}, \quad
\|f\|_{W_V^{s,r}}:=\| \langle \nabla_V \rangle^{s}f \|_{L^r}, \quad
\nabla_V:=\sqrt{\HHH}, 
\]
where $ \langle a \rangle := \sqrt{1 + a^2} $. 
Then we denote $\dot{H}_V^{s}:=\dot{W}_V^{s,2}$ and $H_V^{s}:=W_V^{s,2}$.

Taking this definition, we have the following Lemma \ref{si} and Lemma \ref{equivalence} which show us the equivalence of the standard Sobolev space and the Sobolev spaces associated with $\HHH$.  The Sobolev inequalities and equivalence of Sobolev spaces are crucial and are frequently used in the whole paper.
\begin{lemma}[Sobolev inequalities]\label{si}
For $ d \ge 3 $, let $V: \R^d \rightarrow \R$ satisfy  $\mathbf{(H1)}$ and $\mathbf{(H2)}$. Then it holds that
	\[\|f\|_{L^q} \lesssim \|f\|_{{\dot{W}}_V^{s,p}}, \quad \|f\|_{L^q} \lesssim \|f\|_{W_V^{s,p}},\]
	where $1<p<q<\infty$, $1<p<\frac{d}{s}$, $0 \leq s \leq 2$ and $\frac{1}{q}=\frac{1}{p}-\frac{s}{d}$.
\end{lemma}
\begin{proof}
Let $0 \leq a \leq 1$ and $A_1,A_2>0$. 
According to Theorem 2 in \cite{mt}, we have the following Gaussian heat kernel estimate,
\begin{equation}\label{Gaussian}
0 \leq e^{-t(a+\HHH)}(x,y) \leq \frac{A_1}{t^{d/2}} e^{-A_2\frac{|x-y|^2}{t}} 
\end{equation}
for any $t>0$ and any $x,y \in \R^d$.
Applying \eqref{Gaussian} to 
\[(a+\HHH)^{-\frac{s}{2}}(x,y) = \frac{1}{\Gamma(s)} \int_0^{\infty} t^{\frac{s}{2}-1-\frac{d}{2}}e^{-\frac{|x-y|^2}{t}}dt.\]
Let $t'=  \frac{|x-y|^2}{t}$, we have 
\begin{align*}
    \bigl| (a+\HHH)^{-\frac{s}{2}}(x,y) \bigr| 
    &= \frac{1}{\Gamma(s)} \int_0^{\infty} \biggl( \frac{|x-y|^2}{t'}  \biggr)^{\frac{s-2-d}{2}}e^{-t'}|x-y|^2(t')^{-2}dt'\\
    &=\frac{1}{\Gamma(s)} \int_0^{\infty}|x-y|^{s-d}e^{-t'}(t')^{\frac{d-s}{2}-1}dt'\\
    &= \frac{1}{|x-y|^{d-s}} \biggl( \frac{\Gamma(\frac{d-s}2)}{\Gamma(s)} \biggr) \lesssim \frac{1}{|x-y|^{d-s}}.
\end{align*}
By Lemma \ref{hls}, it implies 
\begin{equation}\label{a+H}
\Vert (a+\HHH)^{-\frac{s}{2}}f \Vert_{L^q} \lesssim \Vert f \Vert_{L^p} 
\end{equation} 
with $\frac{1}{q}=\frac{1}{p}-\frac{s}{d}$. The proof is completed.
In fact, $a=0$ in \eqref{a+H} corresponds to $\|f\|_{L^q} \lesssim \|f\|_{{\dot{W}}_V^{s,p}}$ and $a=1$ in \eqref{a+H} corresponds to $\|f\|_{L^q} \lesssim \|f\|_{{{W}}_V^{s,p}}$.
\end{proof}

\begin{lemma}[Equivalence of Sobolev spaces]\label{equivalence}
For $ d \ge 3 $, let $V: \R^d \rightarrow \R$ satisfy  $\mathbf{(H1)}$ and $\mathbf{(H2)}$. Then it holds that
	\[
	\|f\|_{{\dot{W}}_V^{s,p}} \sim \|f\|_{{\dot{W}}^{s,p}}, \quad \|f\|_{W_V^{s,p}} \sim \|f\|_{W^{s,p}},
	\]
	where $1<p<\frac{3}{s}$ and $0\leq s \leq 2$.
\end{lemma}
\begin{proof}
    Let $0 \leq a \leq 1$. On one hand, using Lemma \ref{holder}, we have 
    \begin{align*}
        \Vert (a+\HHH)f \Vert_{L^r}&\leq \Vert (a-\Delta)f \Vert_{L^r}+\Vert Vf \Vert_{L^r}\\
        &\leq \Vert (a-\Delta)f \Vert_{L^r}+\Vert V \Vert_{L^{\frac{d}{2}}}\Vert f \Vert_{L^{\frac{dr}{d-2r}}}\\
        &\lesssim  \Vert (a-\Delta)f \Vert_{L^r}
    \end{align*}
    and 
    \begin{align*}
        \Vert (a-\Delta)f \Vert_{L^r}&\leq \Vert (a+\HHH)f \Vert_{L^r}+\Vert Vf \Vert_{L^r}\\
        &\leq \Vert (a+\HHH)f \Vert_{L^r}+\Vert V \Vert_{L^{\frac{d}{2}}}\Vert f \Vert_{L^{\frac{dr}{d-2r}}}\\
        &\lesssim \Vert (a+\HHH)f \Vert_{L^r}.
    \end{align*}
    On the other hand, according to the boundness of the imaginary power operator in \cite{aj}, we have 
    \[ \Vert(a+\HHH)^{ib} \Vert_{L^r \to L^r} \lesssim \langle b \rangle^{\frac{d}{2}} \]
    and
    \[ \Vert(a-\Delta)^{ib} \Vert_{L^r \to L^r} \lesssim \langle b \rangle^{\frac{d}{2}} \]
    for any $b\in \R$ and $1<r <\infty$. 
    
    Combining the above two parts, we have 
    \begin{align*}
        &\Vert (a+\HHH)^zf\Vert_{L^r} \lesssim \langle \Im z \rangle^{\frac{d}{2}}\Vert (a-\Delta)^zf\Vert_{L^r},\\
        &\Vert (a-\Delta)^zf\Vert_{L^r} \lesssim \langle \Im z \rangle^{\frac{d}{2}} \Vert (a+\HHH)^zf\Vert_{L^r}
    \end{align*}
    for $1<r<\infty$ when $\Re z =0$ and $1<r<\frac{d}{2}$ when $\Re z =1$. By the Stein-Weiss complex interpolation, we obtain the equivalence
    \[ 
    \Vert (a+\HHH)f\Vert_{L^r} \sim \Vert (a-\Delta)f\Vert_{L^r}, ~1<r<\frac{d}{2}.
    \]
    In particular, $a=0$ corresponds to $\|f\|_{{\dot{W}}_V^{s,p}}$ and $a=1$ corresponds to $\|f\|_{{{W}}_V^{s,p}}$.
\end{proof}

\section{LWP and the  Gorund State}
Hereafter, we focus on the special case $ (\mu,\gamma,d) = (-1,3,5) $ of \eqref{NLHv}. We will firstly introduce the dispersive and Strichartz estimates for the equation \eqref{NLHv}. Having Strichartz estimate, we establish the local well-posedness result of \eqref{NLHv} using the Banach fixed point theory. Finally, we recall the properties of the ground state, which will be used in the proof of scattering theory and blow up result.
\subsection{Strichartz estimates}
A pair $(q,r) \in \Lambda_s$ is denoted as
\begin{align*}
	\frac{2}{q} = 5 (\frac{1}{2}-\frac{1}{r})-s
\end{align*}
where $2 \leq q \leq \infty, ~ 2\leq r < \frac{2d}{d-2}, ~ (q,r,d) \neq (2,\infty, 2)$.
Then we can define the Strichartz norm for any interval $I \subset \R$
\begin{equation*}
	\Vert u \Vert_{S(L^2,I)}:=\sup_{\substack{(q,r)\in \Lambda_0}}\Vert u \Vert_{L^q(I,L^r)}, \quad \Vert v \Vert_{S'(L^2,I)}:=\inf_{\substack{(q,r)\in \Lambda_0}}\Vert v  \Vert_{L^{q'}(I,L^{r'})},
\end{equation*}
where $(q,q')$ and $(r,r')$  are H\"older's conjugate pairs.

Once $ \HHH$ doesn’t have an eigenvalue, then the dispersive estimate holds by Beceanu-Goldberg \cite{BG}. As we all known, the dispersive estimate is the key tool to establish the Strichartz estimate, which is crucial in the nonlinear dispersive equations. 
\begin{lemma}[Dispersive estimate, \cite{dispersive}]\label{dispersive}
	Let $V: \R^5 \rightarrow \R$ satisfy  $\mathbf{(H1)}$ and $\mathbf{(H2)}$. Then it holds that
	\[
	\|e^{-it\HHH}\|_{L^1 \to L^{\infty}} \lesssim |t|^{-\frac{5}{2}}.
	\]
\end{lemma}
 The Schr\"odinger operator $\HHH$ is self-adjoint into $L^2$ because $ \HHH$ doesn’t have an eigenvalue. By Stone’s theorem, the Schr\"odinger evolution group $e^{-it\HHH}$ is generated on $L^2$. Thus we have the following Strichartz estimate combing the dispersive estimate.
\begin{lemma}[Strichartz estimates, \cite{dynamics}]\label{strichartz}
	Let $V: \R^5 \rightarrow \R$ satisfy   $\mathbf{(H1)}$ and $\mathbf{(H2)}$. Then it holds that
	\begin{equation*}
		\begin{aligned}
			\|e^{-it\HHH}f\|_{L^q(\R,L^r)} &\lesssim \|f\|_{L^2}, \\
			\biggl\Vert\int_{0}^{t}e^{-i(t-s)\HHH}f(s)ds\biggl \Vert_{L^q(\R,L^r)} &\lesssim \|f\|_{L^{m'}(\R,L^{n'})},
		\end{aligned}		
	\end{equation*}
	for any $(q,r),(m,n) \in S$, where $(m,m')$ and $(n,n')$ are H\"older's conjugate pairs.
\end{lemma}


\subsection{Local well-posedness} 
In this subsection,we will prove the \eqref{NLHv} is locally well-posed in $H^1$ under the assumptions $\mathbf{(H1)}$ and $\mathbf{(H2)}$ using Banach contraction mapping principle. Here Strichartz estimates plays an important role in the proofs. Besides, the only difference in the proofs is
that the norm equivalence (Lemma \ref{equivalence}) is used, compared with \cite{meng}.

\begin{proposition}[Local well-posedness]\label{lwp}
	Let \ $V: \R^5 \rightarrow \R$ satisfy  $\mathbf{(H1)}$ and $\mathbf{(H2)}$. Then the equation \eqref{NLHv} is locally well-posed in $H^1$ for $(\mu, \gamma, d)=(-1,3,5)$.
\end{proposition}
\begin{proof}
	Our proof is in view of a complete  Banach space $(X,d)$, where
	\begin{align*}X:\{u:\Vert \langle \nabla_V \rangle u\Vert_{S(L^2,I)} \leq M\}\end{align*}
	and
	\[d(u,v):=\Vert u-v \Vert_{S(L^2,I)}.\] $I=[0,T]$ with $T,M>0$ will be chosen later. We demonstrate the integral transformation as
	\begin{equation}\label{transform}
		\Phi(u(t)):=e^{-it\HHH}u_0+i\int_{0}^{t}e^{-i(t-s)\HHH}(|\cdot|^{-3} \ast |u|^2)u(s)ds.
	\end{equation}
	The chief aim is to prove that $\Phi$ is a contraction on $(X,d)$.
	According to Lemma \ref{holder},  Lemma \ref{hls}, and Lemma \ref{strichartz}, we have
	\begin{align*}
			&~\Vert \langle \nabla_V \rangle \Phi (u) \Vert_{S(L^2,I)}\\	
               \leq &~ \Vert e^{-it\HHH}\langle \nabla_V \rangle u_0 \Vert_{S(L^2,I)} + 	\biggl \Vert \int_{0}^{t}e^{-i(t-s)\HHH}\langle \nabla_V \rangle(|\cdot|^{-3} \ast |u|^2)u(s)ds \biggl \Vert_{S(L^2,I)}\\
			\lesssim &~ \Vert\langle \nabla_V \rangle u_0 \Vert_{L_x^2(\R^5)} + 	\Vert \langle \nabla_V \rangle(|\cdot|^{-3} \ast |u|^2)u \Vert_{L_t^2L_x^{\frac{10}{7}}(I \times \R^5)}\\
			\sim &~\Vert u_0 \Vert_{H_x^1(\R^5)}+\biggl\Vert \nabla \biggl(( |\cdot|^{-3} \ast |u|^2)u \biggr)\biggr\Vert_{L_t^2L_x^{\frac{10}{7}}(I \times \R^5)}\\
			\lesssim &~ \Vert u_0 \Vert_{H_x^1(\R^5)}+\biggl\Vert  3\Vert u\Vert_{L_x^{\frac{10}{3}}(\R^5)}^2 \Vert u \Vert_{H_x^1(\R^5)}\biggl\Vert_{L_t^2(I)}\\
			\lesssim &~ \Vert u_0 \Vert_{H_x^1(\R^5)}+3T^{\frac{1}{4}}\Vert u \Vert_{L_t^2L_x^{\frac{10}{3}}(I \times \R^5)}^2 \Vert u \Vert_{L_t^{\infty}H_x^1(I \times \R^5)} \lesssim \Vert u_0 \Vert_{H_x^1(\R^5)}+3T^{\frac{1}{4}}M^3.
	\end{align*}
	Similarly, we have
	\begin{align*}	
			&~d(\Phi(u)-\Phi(v))=\Vert \Phi(u)-\Phi(v) \Vert_{S(L^2, I)}\\
			 \leq &~ \biggl \Vert \int_{0}^{t}e^{-i(t-s)\HHH}[(|\cdot|^{-3} \ast |u|^2)u-(|\cdot|^{-3} \ast |v|^2)v](s)ds \biggl \Vert_{S(L^2,I)}\\
			 \lesssim &~ \Vert (|\cdot|^{-3} \ast |u|^2)u-(|\cdot|^{-3} \ast |v|^2)v \Vert_{S'(L^2,I)}\\
			 \lesssim  &~ 2T^{\frac{1}{4}}\biggl( \Vert u \Vert_{L_t^2L_x^{\frac{10}{3}}(I \times \R^5)}^2+\Vert u+v \Vert_{L_t^2L_x^{\frac{10}{3}}(I \times \R^5)}\Vert v \Vert_{L_t^{\infty}L_x^2(I \times \R^5)} \biggr)\Vert u-v \Vert_{S(L^2,I)}\\
			 \lesssim  &~ 6T^{\frac{1}{4}} \Vert u \Vert_{S(L^2,I)}^2d(u,v) \lesssim 6T^{\frac{1}{4}} M^2d(u,v).
	\end{align*}
	It implies that for any $u,v \in X$, there exists the constant $C>0$ independent of $u_0$ and $T$ such that
	\[\Vert \langle \nabla_V \rangle \Phi (u) \Vert_{S(L^2,I)} \leq C\Vert u_0 \Vert_{H_x^1(\R^5)}+3CT^{\frac{1}{4}}M^3,\]
	\[d(\Phi(u)-\Phi(v))\leq 6CT^{\frac{1}{4}} M^2d(u,v).\]
	If we choose $M=2C\Vert u_0 \Vert_{H_x^1(\R^5)}$ and take $T>0$ sufficiently small, then we have
	\[\Vert \langle\nabla_V \rangle \Phi (u) \Vert_{S(L^2,I)} \leq M , ~ d(\Phi(u)-\Phi(v)) \leq d(u,v). \]
    The proof of Local well-posedness is completed.
\end{proof}

\subsection{Ground state}
In this subsection, we recall the properties of the ground state, which plays a critical role in  understanding the mass-energy threshold of the solution.

Let us start with the elliptic equation
\begin{equation*}
	\Delta Q-Q+(|\cdot|^{-3} \ast |Q|^2)Q=0,
\end{equation*}
where $Q$ is called as the ground state.  And $u(t,x)=e^{it}Q(x)$ is a time-global non-scattering
solution to $$iu_t +\Delta u+(|\cdot|^{-3} \ast |u|^2)u=0$$
and is called a standing wave solution. 
First of all, it follows from \cite{meng} that we can compute the sharp constant  for Gagliardo-Nirenberg inequality of convolution type
\begin{equation}\label{p<c}
	P(u):=\int_{\R^5}\int_{\R^5}\frac{|u(x)^2|u(y)|^2|}{|x-y|^3}dxdy\leq C_{GN}\Vert u \Vert_{L_x^2(\R^5)}\Vert u \Vert_{\dot{H}_x^1(\R^5)}^3
\end{equation}
which is attained at $u(t,x)=e^{it}Q(x)$.
We can compute the ratios between $ E(Q), P(Q) $ and $ M(Q) $ by the standard Pohozaev's identities as in \cite{meng}, as described below.
\begin{lemma}[Relation]
For the ground state $ Q $ of \eqref{elliptic}, we have
    \begin{equation}\label{pq4}
	\Vert Q \Vert_{\dot{H}_x^1(\R^5)}^2=3\Vert Q \Vert_{L_x^2(\R^5)}^2, \quad P(Q)=4\Vert Q \Vert_{L_x^2(\R^5)}^2.
    \end{equation}
    Then, associating \eqref{pq4} with mass \eqref{mass} and energy \eqref{energy} conservation, we have
    \begin{equation}\label{mqeq}
	\begin{aligned}
		M(Q)E(Q)
		=\frac{1}{6}\Vert Q \Vert_{L_x^2(\R^5)}^2\Vert Q \Vert_{\dot{H}_x^1(\R^5)}^2.
	\end{aligned}
\end{equation}
\end{lemma}
 
Thus by the property of \eqref{p<c}, we get
\begin{equation}\label{cgn}
	C_{GN}=\frac{P(Q)}{\Vert Q \Vert_{L_x^2(\R^5)}\Vert Q \Vert_{\dot{H}_x^1(\R^5)}^3}=\frac{4}{3\sqrt{3}\Vert Q \Vert_{L_x^2(\R^5)}^2}.
\end{equation}
Now we show some conclusion about the coercivity condition, which will be used in the proof of scattering theory.
\begin{property}[Coercivity]\label{coer}
Assume that $u:I \times \R^5 \to C$ is the maximal-lifespan solution to \eqref{NLHv} for $(\mu, \gamma
 ,d)=(-1,3,5)$ and  $ V \ge 0 $ satisfies $\mathbf{(H1)}$ and $\mathbf{(H2)}$. If 
 $M(u_0)E(u_0)<(1-\delta)M(Q)E(Q)$
 and
 $\Vert u_0 \Vert_{L_x^2(\R^5)}\Vert \nabla_V u_0 \Vert_{L_x^2(\R^5)} \leq \Vert Q \Vert_{L_x^2(\R^5)}\Vert \nabla Q \Vert_{L_x^2(\R^5)}$.
 \begin{itemize} 
  \item Then there exists $\delta'=\delta'(\delta)>0$. 
 So that for all $t \in I$, we have
     \begin{equation}\label{coer1}
		\Vert u(t) \Vert_{L_x^2(\R^5)}\Vert \nabla_V u(t) \Vert_{L_x^2(\R^5)}\leq (1-\delta')\Vert Q \Vert_{L_x^2(\R^5)}\Vert \nabla Q \Vert_{L_x^2(\R^5)}.
	\end{equation}
	In particular, $I=\R^5$ and $u$ is uniformly bounded in $H_x^1(\R^5)$.
  \item Suppose $\Vert u \Vert_{L_x^2(\R^5)}\Vert \nabla_V u \Vert_{L_x^2(\R^5)}\leq (1-\delta)\Vert Q \Vert_{L_x^2(\R^5)}\Vert \nabla Q \Vert_{L_x^2(\R^5)}$, then there exists $\delta'=\delta'(\delta)>0$ such that
	\begin{equation}\label{coer2}
		\Vert \nabla_V u \Vert_{L_x^2(\R^5)}^2-\frac{3}{4}P(u) \geq \delta'P(u).
	\end{equation}
  \item There exists $R=R(\delta, M(u), Q)>0$ sufficiently large such that
	\begin{equation}\label{coer3}
		\sup_{t \in \R}\Vert \chi_Ru(t)\Vert_{L_x^2(\R^5)}\Vert \nabla_V\chi_Ru(t)\Vert_{L_x^2(\R^5)}<(1-\delta)\Vert Q \Vert_{L_x^2(\R^5)}\Vert \nabla Q \Vert_{L_x^2(\R^5)}.
	\end{equation}
	In particular, by \eqref{coer2}, there exists $\delta'=\delta'(\delta)>0$ so that
	\begin{equation}
		\Vert \nabla_V \chi_Ru(t) \Vert_{L_x^2(\R^5)}^2-\frac{3}{4}P(\chi_Ru(t)) \geq \delta'P(\chi_Ru(t))
	\end{equation}
	uniformly for $t\in\R.$
 \end{itemize}
\end{property}

\begin{proof}
The proof is similar to that of \cite{meng}, hence we here only give the proof of \eqref{coer1}. Set
	\begin{equation*}
		f(t)=\frac{\Vert u \Vert_{L_x^2(\R^5)}\Vert\nabla_V u\Vert_{L_x^2(\R^5)}}{\Vert Q \Vert_{L_x^2(\R^5)}\Vert \nabla Q \Vert_{L_x^2(\R^5)}}
		\sim \frac{\Vert u \Vert_{L_x^2(\R^5)}\Vert u\Vert_{\dot{H}_x^1(\R^5)}}{\Vert Q \Vert_{L_x^2(\R^5)}\Vert  Q \Vert_{\dot{H}_x^1(\R^5)}} \in C(I).
	\end{equation*}
	It follows from \eqref{mass}, \eqref{energy} and \eqref{p<c} that we have
	\begin{align*}
		M(u)E(u)&=\Vert u 
            \Vert_{L_x^2(\R^5)}^2\biggl(\frac{1}{2}\Vert u \Vert_{\dot{H}_x^1(\R^5)}^2+\frac{1}{2}\int_{\R^5}V|u|^2dx-\frac{1}{4}P(u)\biggr)\\
            & \geq \Vert u \Vert_{L_x^2(\R^5)}^2\biggl(\frac{1}{2}\Vert u \Vert_{\dot{H}_x^1(\R^5)}^2-\frac{1}{4}P(u)\biggr)\\
		&\geq \Vert u \Vert_{L_x^2(\R^5)}^2\biggl(\frac{1}{2}\Vert u \Vert_{\dot{H}_x^1(\R^5)}^2-\frac{1}{4}C_{GN}\Vert u \Vert_{L_x^2(\R^5)}\Vert u \Vert_{\dot{H}_x^1(\R^5)}^3\biggr)\\
		&=\frac{1}{2}\Vert u \Vert_{L_x^2(\R^5)}^2\Vert u \Vert_{\dot{H}_x^1(\R^5)}^2-\frac{1}{4}C_{GN}\Vert u \Vert_{L_x^2(\R^5)}^3\Vert u \Vert_{\dot{H}_x^1(\R^5)}^3.
	\end{align*}
	Combining \eqref{mqeq} and \eqref{cgn}, we have
	\begin{align*}
		\frac{M(u)E(u)}{M(Q)E(Q)}&=\frac{\Vert u \Vert_{L_x^2(\R^5)}^2\Vert u\Vert_{\dot{H}_x^1(\R^5)}^2}{2M(Q)E(Q)}-\frac{C_{GN}\Vert u \Vert_{L_x^2(\R^5)}^3\Vert u\Vert_{\dot{H}_x^1(\R^5)}^3}{4M(Q)E(Q)}\\
		&=\frac{3\Vert u \Vert_{L_x^2(\R^5)}^2\Vert u\Vert_{\dot{H}_x^1(\R^5)}^2}{\Vert Q \Vert_{L_x^2(\R^5)}^2\Vert Q \Vert_{\dot{H}_x^1(\R^5)}^2}-\frac{4}{3\sqrt{3}\Vert Q \Vert_{L_x^2(\R^5)}^2} \cdot \frac{3\Vert u \Vert_{L_x^2(\R^5)}^3\Vert u\Vert_{\dot{H}_x^1(\R^5)}^3}{2\Vert Q \Vert_{L_x^2(\R^5)}^2\Vert Q \Vert_{\dot{H}_x^1(\R^5)}^2}\\
		&=3f^2-2f^3 \in (0,1-\delta).
	\end{align*}
	So there exists $0<\rho< 1$ such that $f(t) \in (0,\rho),\forall \ t\in I$, and then \eqref{coer1} holds.\end{proof}

\begin{remark}
If $\delta=0$ in Property \ref{coer}, we will see  $f$ is strictly less than 1. Actually, let $g(f)=3f^2-2f^3$, then $g(0)=0$ and $g(1)=1$, the figure  is shown as Figure.\ref{fig:enter-label}. Under this condition, $g$ is monotonically increasing in (0,1). 
Thus \eqref{lambdatq}  holds. Moreover, we further find that $f>1$ when $g(f(t)) \in (0,1)$ consistently holds. Hence \eqref{lambdat>q} is reasonable.
\begin{figure}[ht]
    \centering
    \includegraphics[scale=0.35]{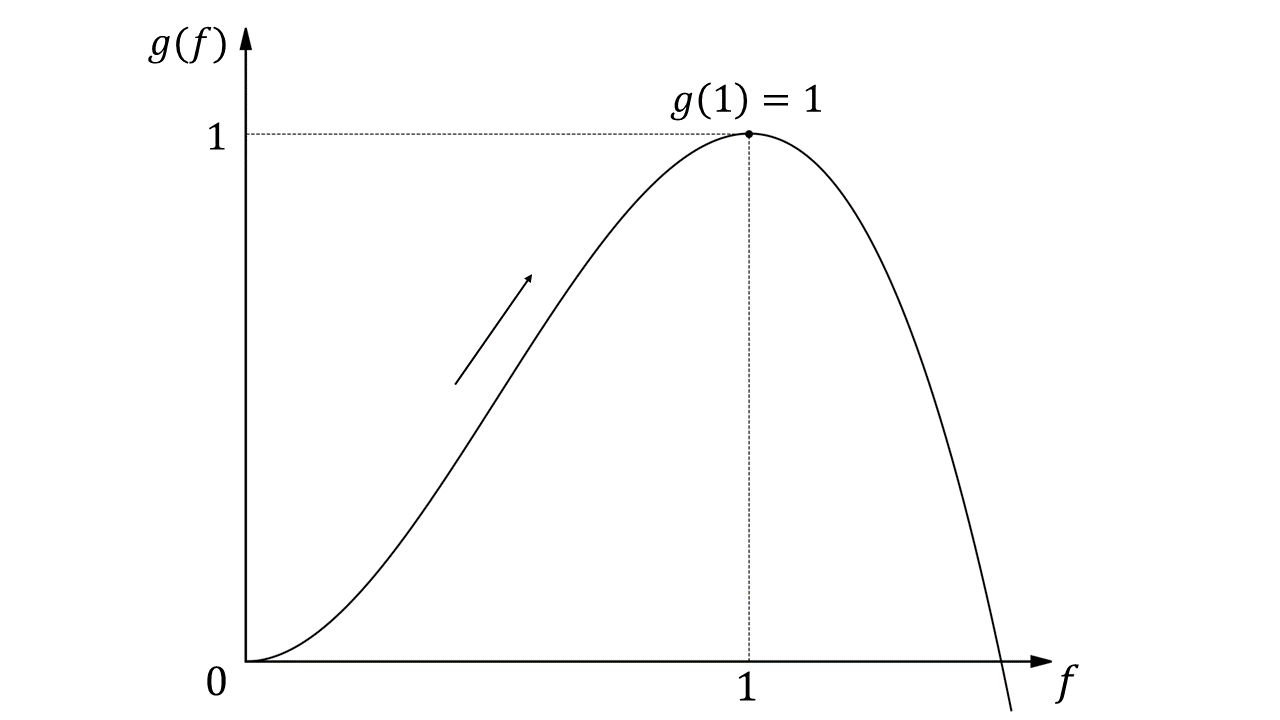}
    \caption{The graph of $g(f)$.}
    \label{fig:enter-label}
\end{figure}

	
\end{remark}

\section{Proof of Scattering}
In this section, we prove scattering criteria firstly. Then we prove the Morawetz estimate, and combine it with scattering criteria to prove the scattering part in Theorem \ref{conclusion}. The proofs are based on the argument of Dodson-Murphy \cite{dodson}.
\subsection{Proof of scattering criteria}
In this subsection, we will prove a scattering criteria. That is, if the solution is bounded in $H^1$, then the solution scatters only if the mass of the solution inside a ball is sufficiently small in some sense. 
\begin{proposition}[Scattering criteria]\label{scattercriterion}
	For $(\mu,\gamma,d) =(-1,3,5)$ in \eqref{NLHv}, let $V: \R^5 \rightarrow \R$ satisfy $\mathbf{(H1)}$ and $\mathbf{(H2)}$. If $u \in H^1$ is a solution to  \eqref{NLHv}   satisfying
	\begin{equation}\label{sup}
		\sup_t \Vert u(t) \Vert_{H^1} :=\mathcal{E} < +\infty,
	\end{equation}
	then there exist two constants $R>0$  and $\eps > 0$, depending only on E, such that if
	\begin{equation}\label{lim}
		\lim_{t \to \infty} \int_{B(0,R)}|u(t,x)|^2dx \leq \eps^{1+},
	\end{equation}
	then u scatters forward in time in $H^1(\R^5)$.
	
\end{proposition}

\begin{proof}
	Our proof is divided into four steps.
	
	$\mathit{Step \ one}$. In this step, we will prove the Strichartz norm of $u(t,x)$ is very small in $[T-l,T]$, where $l$ will be decided later. That is to prove
	\begin{equation}\label{52eps}
		||u(t,x)||_{L_t^{\infty}L_x^{\frac{5}{2}}([T-l,T]\times\R^5)} \lesssim \eps^{\frac{1}{4}+}.
	\end{equation}
	
According to the assumption \eqref{lim}, there exists $T_1>0$  such that
\begin{equation}
\int_{B(0,R)} |u(T_1,x)|^2 dx \leq \eps^{1+}. 
\label{BR}
\end{equation}
Let $\eta_R \in C^{\infty}(\R^5)$, define the radial cut-off function:
\begin{equation}\label{eta}
\eta_R(x)=\left\{
\begin{aligned}
& 1, \qquad & |x| \leq \frac{R}{2},\\
& 0, \qquad & |x| \geq R, \\
\end{aligned}
\right .
\end{equation}
with $ 0 \le \eta_R (x) \le 1 $ and $\frac{2}{R}<|\nabla \eta_R|<\frac{3}{R}$ for $ \frac{R}2 < \vert x \vert < R $. 
Then, by \eqref{BR}, we have 
\[
\int_{\R^5} |\eta_R(x)u(T_1,x)|^2 dx 
\le \Vert \eta_R \Vert_{L^{\infty}} \int_{B(0,R)} |u(T_1,x)|^2 dx 
\leq \eps^{1+}. 
\]
To establish a connection between $u(T_1,x)$ and $\Vert u \Vert_{H_x^1(\R^5)}$, we use the identity 
\[
\partial_t|u|^2 = 2 \Re (\bar{u} u_t) = -2 \Im (\bar{u} \Delta u), 
\] 
derived from \eqref{NLH}, together with integration by parts and Cauchy-Schwarz inequality, we deduce
	\begin{align*}
		& ~\Big\vert \frac{d}{dt} \int_{\R^5}|\eta_R(x)|^2|u(T_1,x)|^2dx \Big\vert \\
		\leq & ~ 2 \int_{\R^5} \eta_R(x) |\nabla\eta_R(x)| |\Im(\bar{u}\Delta u)| dx 
		\leq \frac{6}{R}||u||_{H_x^1(\R^5)}^2,
		\nonumber
	\end{align*}
   which satisfies $\frac{2}{R}<|\nabla \eta_R|<\frac{3}{R}$ when $\frac{R}{2}<|x|<R$.

   On one hand, we can find
   \begin{align*}
  		& ~ ||\eta_R(x)u(t,x)||_{L_t^{\infty}L_x^2([T-l,T]\times\R^5)}
  		\leq  \biggl\| \biggl (\int_{\R^5}|\eta_R(x)|^2|u(t,x)|^2dx\biggl)^\frac{1}{2}\biggl\|_{L_t^{\infty}([T-l,T])}\\
  		\leq & ~ \biggl\| \biggl (\int_{\R^5}|\eta_R(x)|^2|u(T_1,x)|^2dx +\biggl | \frac{d}{dt} \int_{\R^5}|\eta_R(x)|^2|u(T_1,x)|^2dx\biggr| |t-T_1| \biggl)^\frac{1}{2}  \biggl\|_{L_t^{\infty}([T-l,T])}\\
  		\leq & ~ \biggl\| \biggl (\int_{\R^5}|\eta_R(x)|^2|u(T_1,x)|^2dx + \frac{6}{R}||u||_{H_x^1(\R^5)}^2 |t-T_1| \biggl)^\frac{1}{2}  \biggl\|_{L_t^{\infty}([T-l,T])}\\
  		\leq & ~ \biggl\| \biggl (\eps^{1+} + \frac{6}{R}||u||_{H_x^1(\R^5)}^2 |t-T_1| \biggl)^\frac{1}{2}  \biggl\|_{L_t^{\infty}([T-l,T])} \leq 2  \eps^{\frac{1}{2}+},
   	\nonumber
   \end{align*}
   where $l=\eps^{-\frac{1}{8}}$, $R\geq 2||u||_{L_t^{\infty}H_x^1(\R \times \R^5)}\eps^{-\frac{9}{8}-}$, and $t,T_1 \in [T-l,T]$.
   Using Lemma \ref{interpolation} and Sobolev embedding, we have
   \begin{equation}
	\begin{aligned}
		&~||\eta_R(x)u(t,x)||_{L_t^{\infty}L_x^{\frac{5}{2}}([T-l,T]\times\R^5)}\\
		\leq & ~ ||\eta_R(x)u(t,x)||_{L_t^{\infty}L_x^2([T-l,T]\times\R^5)}^{\frac{1}{2}} ||\eta_R(x)u(t,x)||_{L_t^{\infty}L_x^{\frac{10}{3}}([T-l,T]\times\R^5)}^{\frac{1}{2}}\\
		\leq & ~ \sqrt{2}\eps^{\frac{1}{4}+}||u(t,x)||_{L_t^{\infty}H_x^1(\R \times\R^5)}^{\frac{1}{2}}.
	\end{aligned}
	\nonumber
   \end{equation}
   On the other hand, we also have
   \begin{align*}
		&~||(1-\eta_R(x))u(t,x)||_{L_t^{\infty}L_x^{\frac{5}{2}}([T-l,T]\times\R^5)}\\
		\leq & ~ ||(1-\eta_R(x))u(t,x)||_{L_t^{\infty}L_x^{\infty}([T-l,T]\times\R^5)}^{\frac{1}{5}} ||(1-\eta_R(x))u(t,x)||_{L_t^{\infty}L_x^2([T-l,T]\times\R^5)}^{\frac{4}{5}}\\
		\leq & ~ (\frac{4}{R^2})^{\frac{1}{5}}||u(t,x)||_{L_t^{\infty}H_x^1(\R \times\R^5)}^{\frac{4}{5}}
		\leq  \sqrt{2}\eps^{\frac{1}{4}+}||u(t,x)||_{L_t^{\infty}H_x^1(\R \times\R^5)}^{\frac{1}{2}},
   \end{align*}
   where we have used Lemma \ref{rse}. 
   
   Combining the above two parts, we can find
   \begin{equation*}
   		||u(t,x)||_{L_t^{\infty}L_x^{\frac{5}{2}}([T-l,T]\times\R^5)} \leq 2 \sqrt{2}\eps^{\frac{1}{4}+}||u(t,x)||_{L_t^{\infty}H_x^1(\R \times\R^5)}^{\frac{1}{2}} \lesssim \eps^{\frac{1}{4}+}
   \end{equation*}
   for $l=\eps^{-\frac{1}{8}}$.

    $\mathit{Step \ two}$.  In this step, we will prove
    \begin{equation}\label{fh1}
    	\lim_{T \to \infty}\int_{T-l}^{T}\Vert f(u(s)) \Vert_{H_x^1(\R^5)}ds=0
    \end{equation}
    if \eqref{52eps} holds, where $f(u)=- (|\cdot|^{-3} \ast |u|^2)u$.

    Firstly, we can find
    \begin{align*}
    		\nabla f(u)
            =- \bigl(|\cdot|^{-3} \ast (2\Re \bar{u}\nabla u)\bigr)u- \bigl(|\cdot|^{-3} \ast (|u|^2)\bigr)\nabla u.
    \end{align*}
    Using Lemma \ref{holder} and Lemma \ref{hls}, together with Sobolev embedding, we have
    \begin{align*}
   		\Vert f(u) \Vert_{H_x^1(\R^5)} & \leq \Vert |\cdot|^{-3} \ast |u|^2 \Vert_{L_x^2(\R^5)}\Vert u \Vert_{W_x^{1,\frac{10}{3}}(\R^5)} +\Vert |\cdot|^{-3} \ast (2\Re \bar{u}\nabla u) \Vert_{L_x^5(\R^5)}\Vert u \Vert_{L_x^{\frac{10}{3}}(\R^5)}\\
   		&\leq \Vert u \Vert_{L_x^{\frac{10}{3}}(\R^5)}^2 \Vert u \Vert_{W_x^{1,\frac{10}{3}}(\R^5)}+\Vert 2\Re \bar{u}\nabla u \Vert_{L_x^{\frac{5}{3}}(\R^5)}\Vert u \Vert_{L_x^{\frac{10}{3}}(\R^5)}\\
   		&\leq 3\Vert u \Vert_{L_x^{\frac{10}{3}}(\R^5)}^2 \Vert u \Vert_{W_x^{1,\frac{10}{3}}(\R^5)}.
    \end{align*}

    Secondly, we consider the integral with the time in $[T-l,T]$,
    \begin{align}\label{3uuu}
   		\int_{T-l}^{T}  \Vert f(u(s)) \Vert_{H_x^1(\R^5)}ds
   		&\leq 3\int_{T-l}^{T} \Vert u \Vert_{L_x^{\frac{10}{3}}(\R^5)}^2 \Vert u \Vert_{W_x^{1,\frac{10}{3}}(\R^5)}ds\\
   		&\leq 3\Vert u \Vert_{L_t^{\infty}L_x^{\frac{5}{2}}([T-l,T] \times \R^5)}\Vert u \Vert_{L_t^2L_x^5([T-l,T] \times \R^5)} \Vert u \Vert_{L_t^2W_x^{1,\frac{10}{3}}([T-l,T] \times \R^5)}. \notag
    \end{align}

    Thirdly, to find the limit of \eqref{3uuu} when $T \to \infty$, we will  analyze $\Vert u \Vert_{L_t^{\infty}L_x^{\frac{5}{2}}([T-l,T] \times \R^5)}$, $\Vert u \Vert_{L_t^2L_x^5([T-l,T] \times \R^5)}$, $\Vert u \Vert_{L_t^2W_x^{1,\frac{10}{3}}([T-l,T] \times \R^5)}$ respectively. We have already proved \eqref{52eps}, then we only need to prove the other two terms. According to the Strichartz estimates and Sobolev embedding, we can find that
    \begin{equation*}
    	\begin{aligned}
    		&\Vert e^{it\HHH} u \Vert_{L_t^2L_x^5(\R \times \R^5)} \lesssim \Vert e^{it\HHH} u \Vert_{L_t^2W_x^{1,\frac{10}{3}}(\R \times \R^5)} \lesssim \Vert u \Vert_{H_x^1(\R^5)},\\
    		&\Vert e^{it\HHH} u \Vert_{L_t^4L_x^5(\R \times \R^5)} \lesssim \Vert e^{it\HHH} u \Vert_{L_t^4W_x^{1,\frac{5}{2}}(\R \times \R^5)} \lesssim \Vert u \Vert_{H_x^1(\R^5)}.
    	\end{aligned}
    \end{equation*}
    We deduce $\Vert u \Vert_{L_t^qL_x^r([T-l,T] \times \R^5)}\lesssim \langle l \rangle$, where $\langle l \rangle=(1+|l|^2)^{\frac{1}{2}}$ for any $l \in \R$ and
    $$
    \Vert u \Vert_{L_t^qL_x^r(I \times \R^5)}:= \max \bigl\{\Vert u \Vert_{L_t^2W_x^{1,\frac{10}{3}}(I \times \R^5)}, \Vert u \Vert_{L_t^2L_x^5(I \times \R^5)}, \Vert u \Vert_{L_t^4L_x^5(I \times \R^5)} \bigr\}
    $$
    for $I=[T-l,T]$.
    By  Duhamel formula
    \begin{equation}\label{duhamel}
    	u(t)=e^{i(t-T)\HHH}u(T)-i\int_{T}^{t}e^{i(t-s)\HHH}f(u(s))ds
    \end{equation}
    and Lemma \ref{holder}, together with Lemma \ref{equivalence},  we have
    \begin{align*}
    	\Upsilon &:= \Vert u \Vert_{L_t^2L_x^5([T-l,T] \times \R^5)}+\Vert u \Vert_{L_t^4L_x^5([T-l,T] \times \R^5)}+\Vert u \Vert_{L_t^2W_x^{1,\frac{10}{3}}([T-l,T] \times \R^5)}\\
    	& \ \leq 3\biggl( \Vert u(T) \Vert_{H_x^1(\R^5)}+\int_{T-l}^{T}\Vert f(u(s)) \Vert_{H_x^1(\R^5)}ds\biggl)\\
    	& \ \leq 3\bigl( \Vert u(T) \Vert_{H_x^1(\R^5)}+l^{\frac{1}{2}} \Vert u \Vert_{L_t^4L_x^5([T-l,T] \times \R^5)}^2 \Vert u \Vert_{L_t^{\infty}H_x^1([T-l,T] \times \R^5)} \bigl)\\
    	& \ \leq 3\Vert u \Vert_{H_x^1(\R^5)}(1+l^{\frac{1}{2}}\Upsilon^2).
    \end{align*}
    From Lemma \ref{continuity}, we can find that \[\Upsilon \lesssim \langle l \rangle^{\frac{1}{2}} \lesssim \eps^{-\frac{1}{16}},\ for \ l=\eps^{-\frac{1}{8}}.\]
    Combining this with \eqref{52eps} and \eqref{3uuu}, then
    \begin{equation}
    	\int_{T-l}^{T}\Vert f(u(s)) \Vert_{H_x^1(\R^5)}ds \lesssim \eps^{\frac{1}{4}+}\eps^{-\frac{1}{16}}\eps^{-\frac{1}{16}}=\eps^{\frac{1}{8}+},
    \end{equation}
    which means \eqref{fh1} holds.

    $\mathit{Step \ three}$. In this step, we will prove
    \begin{equation} \label{uin4103}
    	u(t,x) \in L_t^4L_x^{\frac{10}{3}}(\left[T,+\infty\right) \times \R^5)
    \end{equation}
    when \eqref{fh1} holds.

    In light of Duhamel formula \eqref{duhamel}, we have
    \begin{equation}
    	\begin{aligned}
    		\Vert u \Vert_{L_t^4L_x^{\frac{10}{3}}(\left[T,+\infty\right) \times \R^5)}
    		&\leq \biggl\Vert e^{i(t-T)\HHH}u(T) \biggl\Vert_{L_t^4L_x^{\frac{10}{3}}(\left[T,+\infty\right) \times \R^5)}\\
    		& \quad +\biggl\Vert \int_{T}^{t}e^{i(t-T)\HHH}(|\cdot|^{-3}\ast |u|^2)u(s)ds \biggl\Vert_{L_t^4L_x^{\frac{10}{3}}(\left[T,+\infty\right) \times \R^5)}.
    	\end{aligned}
    \end{equation}

    We can separate the proof into two parts. Firstly, we need to prove
    \begin{equation}\label{(t-T)H=0}
    	\lim_{T \to +\infty } \bigl \Vert e^{i(t-T)\HHH}u(T) \bigr \Vert_{L_t^4L_x^{\frac{10}{3}}(\left[T,+\infty\right) \times \R^5)} =0.
    \end{equation}
    Using Duhamel formula again, we have
    \begin{equation}\label{0+g+g}
    	e^{i(t-T)\HHH}u(T)=e^{it\HHH}u_0+iG_1+iG_2,
    \end{equation}
    where $I_1:=[0,T-l],I_2:=[T-l,T]$, and
    \begin{equation*}
    	G_j(t):=\int_{I_j}e^{i(t-s)\HHH}(|\cdot|^{-3}\ast |u|^2)u(s)ds, j=1,2.
    \end{equation*}

    \textbf{Estimate on} $e^{it\HHH}u_0$. Owing to \eqref{sup} and Strichartz estimates, we have
    \begin{equation}
    	\lim_{T \to +\infty } \bigl \Vert e^{it\HHH}u_0 \bigr \Vert_{L_t^4L_x^{\frac{10}{3}}(\left[T,+\infty\right) \times \R^5)} =0.
    \end{equation}

    \textbf{Estimate on} $G_1$. For
    \begin{equation}
    	\Vert G_1 \Vert_{L_t^4L_x^{\frac{10}{3}}(\left[T,+\infty\right) \times \R^5)} \leq \Vert G_1 \Vert_{L_t^4L_x^{\frac{5}{2}}(\left[T,+\infty\right) \times \R^5)}^{\frac{3}{4}} \Vert G_1 \Vert_{L_t^4L_x^{\infty}(\left[T,+\infty\right) \times \R^5)}^{\frac{1}{4}},
    \end{equation}
    On the one hand, using Strichartz estimates and Sobolev embedding, we have
    \begin{align*}
   		\Vert G_1 \Vert_{L_t^4L_x^{\frac{5}{2}}(\left[T,+\infty\right) \times \R^5)} &= \biggl\Vert  e^{i(t-T+l)\HHH} \int_{0}^{T-l} e^{i(t-T-s)\HHH}f(u(s))ds \biggr\Vert_{L_t^4L_x^{\frac{5}{2}}(\left[T,+\infty\right) \times \R^5)}\\
   		&\leq \bigl\Vert  e^{i(t-T+l)\HHH} u(T-l) \bigr\Vert_{L_t^4L_x^{\frac{5}{2}}(\left[T,+\infty\right) \times \R^5)}+\bigl\Vert  e^{it\HHH} u_0 \bigr\Vert_{L_t^4L_x^{\frac{5}{2}}(\left[T,+\infty\right) \times \R^5)}\\
   		& \leq 2\Vert u \Vert_{L_t^{\infty}H_x^1(\left[T,+\infty\right) \times \R^5)}.
    \end{align*}
    On the other hand, using  Lemma \ref{holder}  Lemma \ref{hls} and Lemma  \ref{dispersive}, we have
    \begin{align*}
   		&~ \Vert  G_1 \Vert_{L_t^4L_x^{\infty}(\left[T,+\infty\right) \times \R^5)}\\
   		 = & ~ \biggl\Vert \biggl\Vert \int_{0}^{T-l}e^{i(t-s)\HHH}(|\cdot|^{-3}\ast |u|^2)u(s)ds \biggr\Vert_{L_x^{\infty}(\R^5)} \biggr\Vert_{L_t^4(\left[T,+\infty\right) )}\\
   		\leq & ~ \biggl\Vert \int_{0}^{T-l}|t-s|^{-\frac{5}{2}} \Vert (|\cdot|^{-3}\ast |u|^2)u\Vert_{L_x^1(\R^5)}ds \biggr\Vert_{L_t^4(\left[T,+\infty\right) )}\\
   		\leq & ~ \biggl\Vert \int_{0}^{T-l}|t-s|^{-\frac{5}{2}} \Vert u\Vert_{L_x^{\frac{20}{9}}(\R^5)}^2 \Vert u \Vert_{L_x^2(\R^5)} ds \biggr\Vert_{L_t^4(\left[T,+\infty\right) )}\\
   		\leq & ~ \bigl\Vert ~ |t-T+l|^{-\frac{3}{2}}  \bigr\Vert_{L_t^4(\left[T,+\infty\right) )}\Vert u \Vert_{L_t^{\infty}H_x^1(\left[T,+\infty\right) \times \R^5)}^3\\
   		\leq & ~ l^{-\frac{5}{4}}\Vert u \Vert_{L_t^{\infty}H_x^1(\left[T,+\infty\right) \times \R^5)}^3.
    \end{align*}
    Combining the above two parts, we get
    \begin{equation}
    	\lim_{T \to +\infty } \Vert G_1 \Vert_{L_t^4L_x^{\frac{10}{3}}(\left[T,+\infty\right) \times \R^5)} =0,\  for\  l = \eps^{-\frac{1}{8}}.
    \end{equation}

    \textbf{Estimate on} $G_2$. Using Duhamel formula once again, we have
    \begin{align*}
        \biggl\Vert \int_{T-l}^{T} e^{i(t-s)\HHH}f(u(s))ds  \biggr\Vert_{L_t^4L_x^{\frac{10}{3}}}
   		\lesssim \int_{T-l}^{T}  \Vert f(u(s))\Vert_{\dot{W}_x^{\frac{1}{2},2}}ds
   		\leq {\frac{1}{2}} \int_{T-l}^{T}  \Vert f(u(s)) \Vert_{H_x^1}ds.
    \end{align*}
    We have proved \eqref{fh1}, that is to say
    \begin{equation}\label{g2=0}
    	\begin{aligned}
    		\lim_{T \to +\infty }  \Vert G_2 \Vert_{L_t^4L_x^{\frac{10}{3}}(\left[T,+\infty\right) \times \R^5)} =0.
    	\end{aligned}
    \end{equation}
    Combining with \eqref{0+g+g}-\eqref{g2=0}, we can obtain that \eqref{(t-T)H=0} holds.

    Secondly, we need to prove
    \begin{equation}\label{tT4103=0}
    	\lim_{T \to +\infty }\biggl\Vert \int_{T}^{t}e^{i(t-T)\HHH}(|\cdot|^{-3}\ast |u|^2)u(s)ds \biggl\Vert_{L_t^4L_x^{\frac{10}{3}}(\left[T,+\infty\right) \times \R^5)}=0.
    \end{equation}
    Using Sobolev embedding and Strichartz estimates, together with Lemma \ref{equivalence}, we can find
    \begin{align*}
   		& ~ \biggl\Vert  \int_{T}^{t}e^{i(t-T)\HHH}(|\cdot|^{-3}\ast |u|^2)u(s)ds \biggl\Vert_{L_t^4L_x^{\frac{10}{3}}(\left[T,+\infty\right) \times \R^5)}\\
   		 \lesssim & ~ \Vert  (|\cdot|^{-3}\ast |u|^2)u \Vert_{L_t^2\dot{W}_{x(V)}^{\frac{1}{2},\frac{10}{7}}(\left[T,+\infty\right) \times \R^5)}\sim \Vert  (|\cdot|^{-3}\ast |u|^2)u \Vert_{L_t^2\dot{W}_x^{\frac{1}{2},\frac{10}{7}}(\left[T,+\infty\right) \times \R^5)}.
    \end{align*}
    Then we estimate the space part by Lemma \ref{holder}, inequality of arithmetic-geometric mean, Lemma \ref{hls} and Lemma \ref{gn}. It shows that
    \begin{align*}
   		& ~ \Vert  (|\cdot|^{-3}\ast |u|^2)u \Vert_{\dot{W}_x^{\frac{1}{2},\frac{10}{7}}(\R^5)}\\
   		\leq & ~ \Vert (|\cdot|^{-3}\ast |u|^2)u \Vert_{L_x^{\frac{10}{7}}(\R^5)}^{\frac{1}{2}} \Vert(|\cdot|^{-3}\ast |u|^2)u \Vert_{\dot{W}_x^{1,\frac{10}{7}}(\R^5)}^{\frac{1}{2}}\\
   		\leq & ~ \frac{1}{2} \Vert (|\cdot|^{-3}\ast |u|^2)u \Vert_{W_x^{1,\frac{10}{7}}(\R^5)}\\
   		\leq & ~  \frac{1}{2} \bigl( \Vert |u|^2 \Vert_{L_x^{\frac{5}{3}}(\R^5)} \Vert u \Vert_{H_x^1(\R^5)} + \Vert (|\cdot|^{-3}\ast 2\Re (\bar{u} \nabla u))u \Vert_{L_x^{\frac{10}{7}}(\R^5)} \bigr)\\
   		\leq  & ~ \frac{3}{2} \Vert u \Vert_{L_x^{\frac{10}{3}}(\R^5)}^2 \Vert u \Vert_{H_x^1(\R^5)}, \,
    \end{align*}
 which yields 
    \begin{equation*}
    	\begin{aligned}
    		\biggl \Vert \frac{3}{2} \Vert u \Vert_{L_x^{\frac{10}{3}}(\R^5)}^2 \Vert u \Vert_{H_x^1(\R^5)} \biggr \Vert_{L_t^2(\left[T,+\infty\right))}
    		\leq \frac{3}{2} \Vert u \Vert_{L_t^4L_x^{\frac{10}{3}}(\left[T,+\infty\right) \times \R^5)}^2 \Vert u \Vert_{L_t^{\infty}H_x^1(\left[T,+\infty\right) \times \R^5)}.
    	\end{aligned}
    \end{equation*}
    According to \eqref{sup}, we know that \eqref{tT4103=0} holds. Thus \eqref{uin4103} holds.

    $\mathit{Step \ four}$. In this step, we will prove that there exists $u_{+}(x) \in H_x^1(\R^5), s.t.$
    \begin{equation}\label{ut-u+=0}
    	\lim_{t \to +\infty }\Vert u(t)-e^{it\HHH}u_{+} \Vert_{H_x^1(\R^5)}=0,
    \end{equation}
	when \eqref{uin4103} holds. Then, for $(\mu,\gamma,d)=(-1,3,5)$ in \eqref{NLHv}, the solution $u$ is global and scatters.

To prove this completely, we will firstly prove the existence of Cauchy Sequences. 
Using the Strichartz estimates, Lemma \ref{hls} and Lemma \ref{si}, we have
	\begin{align*}
		&~\biggl\Vert \int_{t_1}^{t_2}e^{-is\HHH}(|\cdot|^{-3}\ast |u|^2)u(s)ds \biggl\Vert_{H_x^1(\R^5)}\\
		=& ~ \biggl\Vert \int_{t_1}^{t_2}e^{i(t_2-s)\HHH}(|\cdot|^{-3}\ast |u|^2)u(s)ds \biggl\Vert_{H_x^1(\R^5)}\\
		\lesssim & ~ \Vert  (|\cdot|^{-3}\ast |u|^2)u \Vert_{L_t^2W_{x(V)}^{1,\frac{10}{7}}([t_1,t_2] \times \R^5)}\\
		 \sim & ~ \Vert  (|\cdot|^{-3}\ast |u|^2)u \Vert_{L_t^2W_x^{1,\frac{10}{7}}([t_1,t_2] \times \R^5)}\\
		\leq &~ 3\Vert u \Vert_{L_t^4L_x^{\frac{10}{3}}(\left[t_1,+\infty\right) \times \R^5)}^2 \Vert u \Vert_{L_t^{\infty}H_x^1(\left[t_1,+\infty\right) \times \R^5)},
	\end{align*}
	where we have used the fact that $e^{it\HHH}$ is a unitary group for any time $t$. 	
	Thus
	\begin{equation}\label{t1t2}
		\lim_{t_1,t_2 \to +\infty}\biggl\Vert \int_{t_1}^{t_2}e^{-is\HHH}(|\cdot|^{-3}\ast |u|^2)u(s)ds \biggl\Vert_{H_x^1(\R^5)}=0
	\end{equation}
	holds when \eqref{uin4103} holds.
	
	Set
	\begin{equation*}
		u_{+}=e^{-iT\HHH}u(T)-i\int_{T}^{+\infty}e^{-is\HHH}f(u(s))ds.
	\end{equation*}
	According to
	\begin{equation*}
		\Vert e^{-iT\HHH}u(T) \Vert_{H_x^1(\R^5)} = \Vert u(T) \Vert_{H_x^1(\R^5)} \leq \Vert u \Vert_{L_t^{\infty}H_x^1(\R \times \R^5)} < +\infty,
	\end{equation*}
	and \eqref{t1t2} we have
	\begin{equation*}
		\biggl\Vert \int_{T}^{+\infty}e^{i(t-s)\HHH}f(u(s))ds \biggl\Vert_{H_x^1(\R^5)} = \biggl\Vert \int_{T}^{+\infty}e^{-is\HHH}f(u(s))ds \biggl\Vert_{H_x^1(\R^5)} < +\infty.
	\end{equation*}
	Thus  $u_{+} \in H_x^1(\R^5)$.
	
	Similarly, by the Duhamel formula we have
	\begin{equation*}
		\begin{aligned}
			\Vert u(t)-e^{it\HHH}u_{+} \Vert_{H_x^1(\R^5)}&=\biggl\Vert \int_{t}^{+\infty}e^{i(t-s)\HHH}f(u(s))ds \biggl\Vert_{H_x^1(\R^5)}\\
			&= \biggl\Vert \int_{t}^{+\infty}e^{-is\HHH}f(u(s))ds \biggl\Vert_{H_x^1(\R^5)}
			& \to 0(t \to +\infty).
		\end{aligned}
	\end{equation*}
	Thus \eqref{ut-u+=0} holds.
	
	Finally, combining the above four steps, the proof of Theorem \ref{scattercriterion} is completed only if $\eps>0$ is arbitrarily small.
\end{proof}

\subsection{Morawetz estimate}
 In this subsection, we demonstrate Morawetz estimate. As is well-known, the decay estimates for our solutions can be described by applying the Morawetz estimate. 
Let us start with the following virial identity.
\begin{lemma}[Virial identity, \cite{meng}]\label{viriallemma}
       Let $V:\R^5 \to \R$,
       $\ph:\R^5 \to \R$ be a real function to be chosen later,
       and $u$ be a solution to \eqref{NLHv} with $(\mu,\gamma,d)=
 (-1,3,5)$. Define the function
	\begin{align}\label{vt}
		V_{\ph}(t):=\int \ph|u(t)|^2dx
	\end{align}
	and
	\begin{equation}\label{virial}
		\frac{d}{dt}V_{\ph}(t)=2\Im \int_{\R^5} \nabla \ph \cdot (\bar{u}\nabla u) dx:=M_{\ph}(t).
 	\end{equation}
	Then it holds that
	\begin{equation*}
		\begin{aligned}
			\frac{d}{dt} M_{\ph}(t) = &4 \sum_{k=1}^{5} \sum_{k=1}^5 \Re \int_{\R^5}\overline{u_k} u_j \ph_{kj}dx-\int_{\R^5} \vert u \vert^2\Delta^2\ph dx-2\int_{\R^5}\nabla \ph \cdot \nabla V|u|^2dx\\ &-3\int_{\R^5}\int_{\R^5}\frac{x-y}{|x-y|^5}|u(x)|^2|u(y)|^2 \cdot (\nabla \ph(x)-\nabla \ph (y))dxdy
		\end{aligned}
	\end{equation*}
	where we have used the symmetry of functional displacement.
\end{lemma}

Under Lemma \ref{viriallemma}, we can further find Morawetz estimate, which is fundamental to the proof of scattering.

\begin{proposition}[Morawetz estimate]\label{morawetz}
	For $(\mu,\gamma,d)=(-1,3,5)$ in \eqref{NLHv}, let $V:\R^5 \to \R$ satisfy  $\mathbf{(H1)}$ and $\mathbf{(H2)}$, $V \geq 0$, $x \cdot \nabla V \leq 0$, $x \cdot \nabla V \in L^{\frac{5}{2}}$. Let $u_0 \in H^1$ be radically symmetric and  satisfy \eqref{u0q} and \eqref{lambda0q}. Then for any $T>0$ and any $R\geq R_0$ with $R_0$ as in \eqref{coer3}, the corresponding global solution to \eqref{NLHv} 
 with  satisfies
	\begin{equation*}
		\frac{1}{T}\int_{0}^{T}P(\chi_Ru)dt\lesssim_{\delta,u} \frac{1}{T}+\frac{1}{R}+o_R(1).
	\end{equation*}
\end{proposition}

\begin{proof}
	Let $\ph(x)$ in \eqref{virial} be a hybrid function for fixed $R \gg 1$
	\begin{equation*}
		\ph(x)=\left\{
		\begin{aligned}
			&|x|^2, \quad |x|\leq R,\\
			&3R|x|, \quad |x|\geq 2R,\\
		\end{aligned}
		\right .
	\end{equation*}
	and $\partial_r \ph = \nabla \ph \cdot \frac{x}{|x|}$, which satisfies $|\partial_{\alpha} \ph(x)| \leq C_{\alpha}R|x|^{-|\alpha|+1}$ when $R<|x|<2R$, $\partial_r \ph \geq 0$ and $\partial_r^2 \ph \geq 0$. Under these conditions, the matrix $\ph_{jk}$ is nonnegative.Then we have
	\begin{align}\label{d/dt}
		\frac{d}{dt}M_{\ph}(t) \notag \geq & ~ \biggl( 8\sum_{k=1}^{5}\sum_{j=1}^{5}\Re\int_{|x|\leq R}\overline{u_k}u_j\delta_{kj}dx-6\int_{|x|\leq R}\int_{\R^5}\frac{|u(x)|^2|u(y)|^2}{|x-y|^3}dydx\\ 
          \notag
		&-4\int_{|x|\leq R}x\cdot\nabla V|u(t)|^2dx    \biggr)-\biggl( 2\int_{|x|\geq 2R}\nabla \ph \cdot \nabla V|u(t)|^2dx  \biggr)\\ 
		&+12\sum_{k=1}^{5}\sum_{j=1}^{5}\Re\int_{|x|\geq 2R}\frac{R}{|x|}\biggl[\delta_{jk}-\frac{x_j}{|x|}\frac{x_k}{|x|} \biggr] \overline{u_k}u_jdx\\    \notag 
		&-9\int_{|x|\geq 2R}\int_{\R^5}\frac{|x-y|}{|x-y|^5}|u(x)|^2|u(y)|^2\cdot \biggl(\frac{Rx}{|x|}-\frac{Ry}{|y|}\biggr)dydx\\ \notag
		&+24\int_{|x|\geq 2R}\frac{|u|^2R}{|x|^3}dx- \biggl( 2\int_{R<|x|<2R}\nabla \ph \cdot \nabla V|u(t)|^2dx  \biggr)\\	 \notag		
		&+\int_{R<|x|<2R}\mathcal{O}\biggl(\frac{|u|^2R}{|x|^3}\biggr)dx+\int_{R<|x|<2R}\int_{\R^5}\mathcal{O}\biggl( \frac{|u(x)|^2|u(y)|^2}{|x-y|^3} \biggr)dydx. \notag
	\end{align}
	On one hand, from Proposition \ref{morawetz}, we know the fact $x \cdot \nabla V \leq 0$. Then it holds that
	\begin{equation*}
		-4\int_{|x|\leq R}x\cdot\nabla V|u(t)|^2dx \geq 0.
	\end{equation*}
	On the other hand, owing to the continuity of the hybrid function $\ph(x)$, we can separate the interval $(R,+\infty)$ into $(R,2R)$ and $\left[2R,+\infty\right)$. Using the fact $|\nabla \ph \cdot \nabla V|\leq \dfrac{5}{2} |x\cdot \nabla V|$ and $x \cdot \nabla V \in L^{\frac{5}{2}}$,  the Sobolev embedding implies that
	\begin{equation*}
		\begin{aligned}
			& ~ \biggl| \int_{|x|> R}\nabla \ph \cdot \nabla V|u(t)|^2dx \biggr| = \biggl| \Big( \int_{R<|x|<2R}+\int_{|x|\geq 2R} \Big) \nabla \ph \cdot \nabla V|u(t)|^2dx \biggr|\\
			\lesssim & ~ \int_{|x|> R}|x \cdot \nabla V||u(t)|^2dx 
			\lesssim \Vert x \cdot \nabla V \Vert_{L_{|x|>R}^\frac{5}{2}} \Vert u(t) \Vert_{L^\frac{10}{3}}^2\\
			\lesssim & ~ \Vert x \cdot \nabla V \Vert_{L_{|x|>R}^\frac{5}{2}} \Vert \nabla_V u(t) \Vert_{L^2}^2=o_R(1).
		\end{aligned}
	\end{equation*}
	Then we can simplify \eqref{d/dt} into the following equation:
		\begin{align}
			\frac{d}{dt}M_\ph(t) \geq &  \biggl( 8\sum_{k=1}^{5}\sum_{j=1}^{5}\Re\int_{|x|\leq R}\overline{u_k}u_j\delta_{kj}dx-6\int_{|x|\leq R}\int_{\R^5}\frac{|u(x)|^2|u(y)|^2}{|x-y|^3}dydx \biggr) \nonumber \\
			&  + 12 \sum_{k=1}^{5}\sum_{j=1}^{5}\Re\int_{|x|\geq 2R}\frac{R}{|x|}\biggl[\delta_{jk}-\frac{x_j}{|x|}\frac{x_k}{|x|} \biggr] \overline{u_k}u_jdx+24\int_{|x|\geq 2R}\frac{|u|^2R}{|x|^3}dx \nonumber \\
			& -9\int_{|x|\geq 2R}\int_{\R^5}\frac{|x-y|}{|x-y|^5}|u(x)|^2|u(y)|^2\cdot \biggl(\frac{Rx}{|x|}-\frac{Ry}{|y|}\biggr)dydx \nonumber \\
			& +\int_{R<|x|<2R}\mathcal{O}\biggl(\frac{|u|^2R}{|x|^3}\biggr)dx+\int_{R<|x|<2R}\int_{\R^5}\mathcal{O}\biggl( \frac{|u(x)|^2|u(y)|^2}{|x-y|^3} \biggr)dydx + o_R(1) \nonumber \\ 
                := & ~ (I) + (II) + (III) + (IV) + (V) + (VI) + o_R(1). \label{d/dtnew}
		\end{align}
	We will make full use of Lemma \ref{rse} to estimate every terms in \eqref{d/dtnew}.
	\item[$\bullet$] For $ (I) $, according to Property \ref{coer}, we decompose the  integral region and get
	\begin{align*}
		(I)&=8\sum_{k=1}^{5}\sum_{j=1}^{5}\Re\int_{|x|\leq R}\overline{u_k}u_j\delta_{kj}dx-6\int_{|x|\leq R}\int_{\R^5}\frac{|u(x)|^2|u(y)|^2}{|x-y|^3}dydx\\
		&=8\int_{|x|\leq R}|\nabla u|^2dx-6\int_{|x|\leq \frac{R}{2}}\int_{|y|\leq \frac{R}{2}}\frac{|u(x)|^2|u(y)|^2}{|x-y|^3}dydx\\
		&\quad -6\int_{\frac{R}{2}<|x|\leq R}\int_{|y|\leq \frac{R}{2}}\frac{|u(x)|^2|u(y)|^2}{|x-y|^3}dydx-6\int_{|x|\leq R}\int_{|y|> \frac{R}{2}}\frac{|u(x)|^2|u(y)|^2}{|x-y|^3}dydx\\
		&\leq 8 \biggl[  \Vert \nabla_V \chi_Ru(t) \Vert_{L_x^2(\R^5)}^2-\frac{3}{4}P(\chi_Ru(t)) \biggr]\\
		&\quad -6\int_{\frac{R}{2}<|x|\leq R}\int_{|y|\leq \frac{R}{2}}\frac{|u(x)|^2|u(y)|^2}{|x-y|^3}dydx-6\int_{|x|\leq R}\int_{|y|> \frac{R}{2}}\frac{|u(x)|^2|u(y)|^2}{|x-y|^3}dydx\\
		&\lesssim \delta'P(\chi_Ru)-(I)'-(I)'',
	\end{align*}
	where
	\[(I)'=6\int_{\frac{R}{2}<|x|\leq R}\int_{|y|\leq \frac{R}{2}}\frac{|u(x)|^2|u(y)|^2}{|x-y|^3}dydx\]
	and
	\[(I)''=6\int_{|x|\leq R}\int_{|y|> \frac{R}{2}}\frac{|u(x)|^2|u(y)|^2}{|x-y|^3}dydx.\]
	We can compute
	\begin{align*}
		(I)'&=6\int_{\frac{R}{2}<|x|\leq R}\int_{|y|\leq \frac{R}{2}}\frac{|u(x)|^2|u(y)|^2}{|x-y|^3}dydx\\
		&\lesssim \frac{C}{R^2}\Vert u \Vert_{L_x^2(\R^5)}^\frac{1}{2}\Vert \nabla u \Vert_{L_x^2(\R^5)}^\frac{1}{2}\int_{\frac{R}{2}<|x|\leq R}\int_{|y|\leq \frac{R}{2}}\frac{|u(x)||u(y)|^2}{|x-y|^3}dydx\\
		&\lesssim \frac{C}{R^2}\Vert u \Vert_{L_x^2(\R^5)}^\frac{3}{2}\Vert \nabla u \Vert_{L_x^2(\R^5)}^\frac{1}{2}\biggl \Vert \int_{\R^5}\frac{|u(y)|^2}{|x-y|^3}dy\biggr \Vert_{L_x^2(\R^5)}\\
		&\lesssim \frac{C}{R^2}\Vert u \Vert_{L_x^2(\R^5)}^\frac{3}{2}\Vert \nabla u \Vert_{L_x^2(\R^5)}^\frac{1}{2}\Vert u \Vert_{L^{\frac{20}{9}}(\R^5)}^2\\
		&\lesssim \frac{C}{R^2}\Vert u \Vert_{L_x^2(\R^5)}^3\Vert \nabla u \Vert_{L_x^2(\R^5)}.
	\end{align*}
	Likewise,
	\[(I)''	\leq \frac{C}{R^2}\Vert u \Vert_{L_x^2(\R^5)}^3\Vert \nabla u \Vert_{L_x^2(\R^5)}.\]
	
	\item[$\bullet$] For $(II)$, because of $\slashed{\nabla}u=\nabla u-\frac{x}{|x|}\partial_ru$, and $\partial_ru=\frac{x}{|x|}\nabla u$, we have
	\[ \slashed{\nabla}u=\nabla u-\frac{x}{|x|}\frac{x}{|x|}\nabla u .\]
	Thus
	\[|\slashed{\nabla}u|^2=|\nabla u|^2 - \sum_{j=1}^{5}\sum_{k=1}^{5}\frac{x_j}{|x|}\frac{x_k}{|x|}u_j\overline{u_k}.\]
	Considering that $u$ is radial, then
	\begin{equation*}
		\begin{aligned}
			(II)&=4\sum_{k=1}^{5}\sum_{j=1}^{5}\Re\int_{|x|\geq 2R}\frac{R}{|x|}\biggl[\delta_{jk}-\frac{x_j}{|x|}\frac{x_k}{|x|} \biggr] \overline{u_k}u_jdx\\
			&=4\Re\int_{|x|\geq 2R}\frac{R}{|x|}\biggl[ |\nabla u|^2 - \sum_{j=1}^{5}\sum_{k=1}^{5}\frac{x_j}{|x|}\frac{x_k}{|x|}u_j\overline{u_k} \biggr]dx\\
			&=4\int_{|x|\geq 2R}\frac{R}{|x|}|\slashed{\nabla}u|^2=0.
		\end{aligned}
	\end{equation*}
	
	\item[$\bullet$] For $(III)$, $(IV)$, $(V)$, $(VI)$, we can similarly find 
 \begin{align*}
    & (III) =  \int_{|x|\geq 2R}\frac{|u|^2R}{|x|^3}dx \geq 0,\\
    & |(IV)|\leq  \frac{C}{R}\Vert u \Vert_{L_x^2(\R^5)}^2 \Vert \nabla u \Vert_{L_x^2(\R^5)}^2,\\
      &(V) =  \int_{R<|x|<2R}\mathcal{O}\biggl(\frac{|u|^2R}{|x|^3}\biggr)dx \lesssim \frac{1}{R^2}\Vert u \Vert_{L_x^2(\R^5)},\\
       & |(VI)| \leq\frac{C}{R^2}\Vert u \Vert_{L_x^2(\R^5)}^3 \Vert \nabla u \Vert_{L_x^2(\R^5)}.
 \end{align*}
	
	Combining above inequalities and \eqref{d/dtnew}, we discard nonnegative terms and deduce
	\begin{equation*}
		\begin{aligned}
			\frac{d}{dt}M_\ph(t)&=(I) + (II) + (III) + (IV) + (V) + (VI) + o_R(1)\\
			&\gtrsim \delta'P(\chi_Ru)-(I)'-(I)''+(IV)+(V)+(VI)+o_R(1),
		\end{aligned}
	\end{equation*}
	which implies
	\begin{equation*}
		\begin{aligned}
			\delta'P(\chi_Ru) &\lesssim \frac{d}{dt}(M_\ph(t))+(I)'+(I)''+|(IV)|+|(V)|+|(VI)|+o_R(1)\\
			&\lesssim \frac{d}{dt}(M_\ph(t))+\frac{1}{R^2}+\frac{1}{R}+o_R(1)\\
			&\lesssim \frac{d}{dt}(M_\ph(t))+\frac{1}{R}+o_R(1).
		\end{aligned}
	\end{equation*}
	The fundamental theorem of calculus tells us
	\[\delta'\int_{0}^{T}P(\chi_Ru)dt \lesssim \sup_{t \in [0,T]}|M(t)|+\frac{T}{R}+T \cdot o_R(1),\]
	that is
	\[\frac{1}{T} \int_{0}^{T}P(\chi_Ru)dt \lesssim_{\delta,u} \frac{1}{T}+\frac{1}{R}+o_R(1).\]
\end{proof}

\begin{remark}
    Based on the condition that $u_0$ is radial, we find the classical Morawetz estimate to prove the boundness of Strichartz norm, which is crucial for the proof of scattering. Actually, we can employ the interaction Morawetz estimate to the proof when considering the non-radial case. 
\end{remark}

\section{Proof of Blow Up}
In this section, we will prove blow up result.

Firstly, using Hardy inequality
\[\int_{0}^\infty\biggl[\frac{1}{x}\int_{0}^xf(\xi)d\xi\biggr]^pdx\leq \biggl(\frac{p}{p-1}\biggr)^p\int_{0}^\infty f^p(x)dx, \quad f(x)\geq 0, \quad p>1  \]
and the conservation of mass, we have
\begin{equation*}
	\begin{aligned}
		\Vert u_0 \Vert_{L^2(\R^5)}^2=\int_{\R^5}|u(x,t)|^2dx=\int_{\R^5}|x||u|\cdot \frac{|u|}{|x|}dx\leq \Vert xu(t) \Vert_{L^2}\Vert u(t)\Vert_{\dot{H}^1}.
	\end{aligned}
\end{equation*}
It is worth mention that if $u_0 \in L^2(|x|^2dx)$, then the corresponding solution belongs to $L^2(|x|^2dx)$. If we want to prove the solution blow up in finite time,  thus we only need to prove
\begin{equation}\label{blowup1}
	\Vert xu(t) \Vert_{L^2}^2 \to 0 \quad(t\to T^*).
\end{equation}
According to  the classical argument of Glassey \cite{glassey}, if we want to prove \eqref{blowup1},  then we only need to prove\begin{equation}\label{blowup2} \frac{d^2}{dt^2} \Vert xu(t) \Vert_{L^2}^2<0\end{equation} for all $t \in \left[0,T^*\right)$.

Thus, before to prove the blow up part in Theorem \ref{conclusion}, we should consider the quantity $\frac{d^2}{dt^2} \Vert xu(t) \Vert_{L^2}^2$. Therefore, we first give the following lemma.

\begin{lemma}\label{kulemma}
	In case of $\ph (x) =|x|^2$ in \eqref{virial}, we have
	\begin{equation}
		\frac{d}{dt} \Vert xu(t) \Vert_{L^2}^2= M_{|x|^2}(t),\quad \frac{d^2}{dt^2} \Vert xu(t) \Vert_{L^2}^2=\frac{d}{dt} M_{|x|^2}(t)=8K(u(t)),
	\end{equation}
where 	
\begin{equation}\label{ku}
		K(u(t))=\Vert \nabla u(t) \Vert_{L^2}^2-\frac{1}{2} \int_{\R^5}x\cdot \nabla V|u(t)|^2dx-\frac{3}{4} P(u),
\end{equation}
where $P(u)=\int_{\R^5}\int_{\R^5}\frac{|u(x)|^2|u(y)|^2}{|x-y|^3}dxdy$.
Thus there exists some $\delta>0$ satisfying
\begin{equation}\label{kdelta}
		\sup_{t \in \left[0,T^*\right)} K(u(t)) \leq - \delta
	\end{equation} for all t in the existence time.
\end{lemma}
\begin{proof}
	In order to define $\delta$, we need to find the relation between $K(u(t))$ and the conservation of mass and energy. Multiplying $K(u(t))$ with $M(u(t))$ and using the assumption $2V+x\cdot \nabla V \geq 0$, we have
	\begin{equation}
		\begin{aligned}
			K(u(t))M(u(t))&=\biggl(\Vert \nabla u(t) \Vert_{L^2}^2-\frac{1}{2} \int_{\R^5}x\cdot \nabla V|u(t)|^2dx-\frac{3}{4} P(u)\biggr)\Vert u(t) \Vert_{L^2}^2\\
			&\leq \biggl(\Vert \nabla u(t) \Vert_{L^2}^2+ \int_{\R^5} V|u(t)|^2dx-\frac{3}{4} P(u)\biggr)\Vert u(t) \Vert_{L^2}^2\\
			&=\biggl(\Vert \nabla_V u(t) \Vert_{L^2}^2-\frac{3}{4} P(u)\biggr)\Vert u(t) \Vert_{L^2}^2\\
			&=3E(u)M(u)-\frac{1}{2}\Vert \nabla_V u(t) \Vert_{L^2}^2\Vert u(t) \Vert_{L^2}^2\\
			&=3E(u_0)M(u_0)-\frac{1}{2}\biggl(\Vert \nabla_V u(t) \Vert_{L^2}\Vert u(t) \Vert_{L^2} \biggr)^2
		\end{aligned}
	\end{equation}
	for all $t$ in the existence time. By \eqref{u0q}, there exists $\theta=\theta(u_0,Q)>0$ such that \[ E(u_0)M(u_0) <(1-\theta)E(Q)M(Q). \]
	Combining with \eqref{lambda0>q} and \eqref{mqeq}, we have
	\begin{equation}
		\begin{aligned}
			K(u(t))M(u(t))
			&\leq 3E(u_0)M(u_0)-\frac{1}{2}\biggl(\Vert \nabla_V u(t) \Vert_{L^2}\Vert u(t) \Vert_{L^2} \biggr)^2\\
			&\leq 3(1-\theta)\cdot \frac{1}{6}||\nabla u_0||_{L^2}^2||u_0||_{L^2}^2-\frac{1}{2}||\nabla u_0||_{L^2}^2||u_0||_{L^2}^2\\
			&=-\frac{1}{2}\theta||\nabla u_0||_{L^2}^2||u_0||_{L^2}^2
		\end{aligned}
	\end{equation}
	 for all $t$ in the existence time. Then we obtain
	 \begin{equation*}
	 	K(u(t))\leq -\frac{1}{2}\theta||\nabla u_0||_{L^2}^2\biggl(\frac{M(Q)}{M(u_0)}\biggr) =:-\delta
	 \end{equation*}
     for all $t$ in the existence time  which proves \eqref{kdelta}.
\end{proof}
\begin{remark}
    It follows from Lemma \ref{kulemma} that if $u_0 \in L^2(|x|^2dx)$, then  \eqref{blowup2} holds naturally. Thus, the corresponding solution must blow up in finite time.  However,   \eqref{blowup2} 
 can not be used directly if $u_0 \in H^1$.  Therefore, we need to think matters over carefully. In fact, we have to change the choice of  the function  $\ph (x)$. 
\end{remark}
 Now, we give the blow up criteria under the initial data $u_0 \in H^1$ instead of  $u_0 \in L^2(|x|^2dx)$.

\begin{proposition}[Blow up criteria]\label{blowupcriteria}
	For $(\mu,\gamma,d)=(-1,3,5)$ in \eqref{NLHv}, let $V: \R^5 \rightarrow \R$ satisfies  $\mathbf{(H1)}$ and $\mathbf{(H2)}$, and in addition $V \geq 0, \  x \cdot  \nabla V \in L^{\frac{5}{2}}, \  2V+ x\cdot \nabla V \geq 0$. Let $u:\left[0,T^*\right) \times \R^5 \to C$ be a $H^1$ maximal solution to \eqref{NLHv}. Assume that there exists $\delta >0 $ such that
	$		\sup_{t \in \left[0,T^*\right)} K(u(t)) \leq - \delta ,$
	where $K(u(t))$ will be proved later. Then either $T^*<+\infty$ or $T^*=+\infty$, there exists a time sequence $t_n \to +\infty$ such that $$\lim_{n \to +\infty } \Vert \nabla u(t_n)\Vert_{L^2}=\infty.$$
\end{proposition}

\begin{proof}

	If $T^* <+\infty$, by alternative theory,  we are done. If $T^*=+\infty$, then assume by contradiction that
	\begin{equation*}
		\sup_{t \in \left[0,+\infty\right)}\Vert \nabla u(t) \Vert_{L^2} <\infty.
	\end{equation*}
	 Let $\theta: \left[0,+\infty\right) \to [0,1]$ be a smooth function satisfying
	 \begin{equation}\label{thetar}
	 	\theta(r)=\left\{
	 	\begin{aligned}
	 		&0, \quad 0\leq r \leq \frac{1}{2},\\
	 		&1, \quad r\geq 1.\\
	 	\end{aligned}
	 	\right .
	 \end{equation}
	 For $R>0$, we define the radial function that $\ph_R(x)=\ph_R(r):=\theta(r/R)$ with $ r=|x|$ and satisfies $\nabla \ph_R(x)=\frac{x}{rR}\theta'(r/R)$ and $\Vert \nabla \ph_R \Vert_{L^\infty} \lesssim R^{-1}$.
	 Owing to Lemma \ref{viriallemma} and Lemma \ref{kulemma}, we have
	 \begin{align}\label{d2dt2}
 		& ~ \frac{d^2}{dt^2}V_{\ph_R}(t)=\frac{d}{dt}M_{\ph_R}(t)\notag \\
 		=& ~ -\int_{\R^5} \Delta^2\ph_R|u(t)|^2dx + 4\int_{\R^5}  \frac{\ph_R'(r)}{r}|\nabla u|^2dx \notag \\
 		&+4\int_{\R^5} \biggl(\frac{\ph''_R(r)}{r^2}-\frac{\ph_R'(r)}{r^3}\biggr)|x\cdot \nabla u|^2dx-2\int_{\R^5}  \frac{\ph_R'(r)}{r}x\cdot \nabla V|u|^2dx \notag \\
 		&-3\int_{\R^5}  \int_{\R^5} \frac{x-y}{|x-y|^5}|u(x)|^2|u(y)|^2\cdot \biggl(\frac{\ph_R'(r)}{r}x-\frac{\ph_R'(r)}{r}y\biggr)dxdy \notag \\
 		= & ~ 8K(u(t))-8\Vert \nabla u\Vert_{L^2(\R^5)}^2 + 4\int_{\R^5}x\cdot \nabla V|u(t)|^2dx \\
 		&+ 6\int_{\R^5} \int_{\R^5} \frac{|u(x)|^2|u(y)|^2}{|x-y|^3}dxdy-\int_{\R^5} \Delta^2\ph_R|u(t)|^2dx + 4\int_{\R^5}  \frac{\ph_R'(r)}{r}|\nabla u|^2dx \notag \\
 		&+4\int_{\R^5} \biggl(\frac{\ph''_R(r)}{r^2}-\frac{\ph_R'(r)}{r^3}\biggr)|x\cdot \nabla u|^2dx-2\int_{\R^5}  \frac{\ph_R'(r)}{r}x\cdot \nabla V|u|^2dx \notag \\
 		&-3\int_{\R^5}  \int_{\R^5} \frac{x-y}{|x-y|^5}|u(x)|^2|u(y)|^2\cdot \biggl(\frac{\ph_R'(r)}{r}x-\frac{\ph_R'(r)}{r}y\biggr)dxdy, \notag
	 \end{align}
	 where we have used
	 \begin{gather*}
	 	\partial_j=\frac{x_j}{r}\partial_r, \ \partial_{jk}^2=\biggl(\frac{\delta_{jk}}{r}-\frac{x_jx_k}{r^3}\biggr)\partial_r+\frac{x_jx_k}{r^2}\partial_r^2,\\
	 	\begin{aligned}
	 		&\sum_{j,k=1}^{5}\int \partial_{jk}^2\ph_R \Re(\partial_j\bar{u}\partial_ku)dx=\int\frac{\ph_R'(r)}{r}|\nabla u|^2dx
	 	+\int \biggl(\frac{\ph_R''(r)}{r^2}-\frac{\ph_R'(r)}{r^3}\biggr)|x\cdot \nabla u|^2dx.
	 	\end{aligned}
	 \end{gather*}
	 According to the fact that $\ph_R(r)=r^2$ on $r=|x|<R$, we have
	 \begin{align*}
 		&-8\Vert \nabla u\Vert_{L^2(|x|<R)}^2 + 4\int_{|x|<R}x\cdot \nabla V|u(t)|^2dx + 6\int_{|x|<R} \int_{|y|<R} \frac{|u(x)|^2|u(y)|^2}{|x-y|^3}dxdy\\
 		&-\int_{|x|<R} \Delta^2\ph_R|u(t)|^2dx + 4\int_{|x|<R}  \frac{\ph_R'(r)}{r}|\nabla u|^2dx\\
 		&+4\int_{|x|<R} \biggl(\frac{\ph''_R(r)}{r^2}-\frac{\ph_R'(r)}{r^3}\biggr)|x\cdot \nabla u|^2dx-2\int_{|x|<R}  \frac{\ph_R'(r)}{r}x\cdot \nabla V|u|^2dx\\
 		&-3\int_{|x|<R}  \int_{|x|<R} \frac{x-y}{|x-y|^5}|u(x)|^2|u(y)|^2\cdot \biggl(\frac{\ph_R'(r)}{r}x-\frac{\ph_R'(r)}{r}y\biggr)dxdy\\
 		=&-8\Vert \nabla u\Vert_{L^2(|x|<R)}^2 + 4\int_{|x|<R}x\cdot \nabla V|u(t)|^2dx + 6\int_{|x|<R} \int_{|y|<R} \frac{|u(x)|^2|u(y)|^2}{|x-y|^3}dxdy\\
 		&+4\int_{|x|<R}2|\nabla u|^2dx+ 4\int_{|x|<R}\biggr(\frac{2}{r^2}-\frac{2r}{r^3}\biggl)|x\cdot \nabla V|^2dx-2\int_{|x|<R}\frac{2r}{r}x\cdot \nabla V|u(t)|^2dx\\
 		&-3\int_{|x|<R}\int_{|y|<R}\frac{x-y}{|x-y|^5}|u(x)|^2|u(y)|^2\cdot \biggl(\frac{2r}{r}x-\frac{2r}{r}y\biggr)dxdy\\
 		= & ~ 0.	 		
	 \end{align*}
	 Then we can simplify \eqref{d2dt2} into the following equality
	 \begin{align*}
 		\frac{d^2}{dt^2}V_{\ph_R}(t)=& ~ 8K(u(t))-8\Vert \nabla u\Vert_{L^2(|x|>R)}^2\\ 
        &+ 4\int_{|x|>R}x\cdot \nabla V|u(t)|^2dx + 6\int_{|x|>R} \int_{\R^5} \frac{|u(x)|^2|u(y)|^2}{|x-y|^3}dxdy\\
        &-\int_{|x|>R} \Delta^2\ph_R|u(t)|^2dx + 4\int_{|x|>R}  \frac{\ph_R'(r)}{r}|\nabla u|^2dx\\
 		&+4\int_{|x|>R} \biggl(\frac{\ph''_R(r)}{r^2}-\frac{\ph_R'(r)}{r^3}\biggr)|x\cdot \nabla u|^2dx-2\int_{|x|>R}  \frac{\ph_R'(r)}{r}x\cdot \nabla V|u|^2dx\\
 		&-3\int_{|x|>R} \int_{\R^5} \frac{x-y}{|x-y|^5}|u(x)|^2|u(y)|^2\cdot \biggl(\frac{\ph_R'(r)}{r}x-\frac{\ph_R'(r)}{r}y\biggr)dxdy.
	 \end{align*}
	 Using the Cauchy-Schwarz inequality $|x \cdot \nabla u | \geq |x||\nabla u|=r|\nabla u|$ and $\ph_R''(r)\geq 2$, we get that
	 \begin{align*}
        & ~ 4\int_{|x|>R}  \frac{\ph_R'(r)}{r}|\nabla u|^2dx +4\int_{|x|>R} \biggl(\frac{\ph''_R(r)}{r^2}-\frac{\ph_R'(r)}{r^3}\biggr)|x\cdot \nabla u|^2dx-8\Vert \nabla u\Vert_{L^2(|x|>R)}^2\\
	  \leq & ~ 4\int_{|x|> R}\biggl( \frac{\ph_R'(r)}{r}-2 \biggr)|\nabla u(t)|^2dx + 4\int_{|x|> R}\frac{1}{r^2}\biggl(2-\frac{\ph_R'(r)}{r}\biggr)|x \cdot \nabla u(t)|^2dx \leq 0.	 	
    \end{align*}
	 Using the fact $\frac{\ph_R'(r)}{r} \leq 2$ and  Sobolev embedding, we get that
	 \begin{align*}
	   & ~ \biggr| 4 \int_{|x|> R} x\cdot \nabla V|u(t)|^2dx - 2\int_{|x|>R}\frac{\ph_R'(r)}{r} x \cdot \nabla V|u(t)|^2dx \biggl|\\
	   \leq & ~ 8\int_{|x|>R} |x \cdot \nabla V||u(t)|^2dx
	  \leq 8\Vert x \cdot \nabla V \Vert_{L^{\frac{5}{2}}(|x|> R)}\Vert u(t) \Vert_{L^{\frac{10}{3}}}^2\\
	  \lesssim  & ~ \Vert x\cdot \nabla V \Vert_{L^{\frac{5}{2}}(|x|> R)} \Vert u(t) \Vert_{H^1}^2  =  o_R(1)\Vert u(t) \Vert_{H^1}^2.
	 \end{align*}
	 Moreover, using Lemma \ref{holder}, Lemma \ref{hls} and Sobolev embedding, we have
	 \begin{align*}\label{6pu-3pu}
 		&~ \biggl| 6\int_{|x|>R} \int_{\R^5}\frac{|u(x)|^2|u(y)|^2}{|x-y|^3}dxdy - 3\int\int \frac{x-y}{|x-y|^5}|u(x)|^2|u(y)|^2\cdot \frac{\ph_R'(r)}{r}(x-y)dxdy \biggr|\\
 		\leq & ~ 12\biggl|\int_{|x|>R} \int_{\R^5}\frac{|u(x)|^2|u(y)|^2}{|x-y|^3}dxdy \biggr|
            \lesssim  \Vert u \Vert_{L^{\frac{5}{2}}(|x|>R)}^2 \Vert u \Vert_{H^1}^2 
            \lesssim \Vert u \Vert_{L^2(|x|>R)} \Vert u \Vert_{H^1}^3	. 	
	 \end{align*}
	 Thus, we have
	 \begin{equation}\label{d2v<u}
	 	\begin{aligned}
	 		\frac{d^2}{dt^2}V_{\ph_R}(t) \leq 8K(u(t))+o_R(1)\Vert u(t) \Vert_{H^1}^2 + C\Vert u(t) \Vert_{L^2(|x|>R)}\Vert u(t) \Vert_{H^1}^3.
	 	\end{aligned}
	 \end{equation}
	 For $(\mu,\gamma,d)=(-1,3,5)$ in \eqref{NLHv}. Assume that $u\in C(\left[0,+\infty\right),H^1)$ is a solution to \eqref{NLHv} satisfying
	 \begin{equation}\label{supdelta<infty}
	 	\sup_{t \in \left[0,+\infty\right)}\Vert \nabla u(t) \Vert_{L^2} < \infty.
	 \end{equation}
	 Then for $V_{\ph_R}(t)=\int \ph_R|u(t)|^2dx$  and $\theta$ in \eqref{thetar}, we have 
  \begin{align*}
\int_{|x|\geq R} |u(x,t)|^2 \, dx &\leq \int_{|x|\geq R} |u(x,0)|^2 \, dx + \int_0^t \left| 2 \int_{|x|\geq R} \nabla u(x,s) \cdot \nabla \overline{u(x,s)} \, dx \right| \, ds \\
&\leq \int_{|x|\geq R} |u(x,0)|^2 \, dx + 2 \int_0^t \left( \int_{|x|\geq R} |\nabla u(x,s)|^2 \, dx \right) \, ds \\
&\leq \int_{|x|\geq R} |u(x,0)|^2 \, dx + \frac{C}{R} t\\
& \leq  V_{\ph_R}(0) + \frac{C}{R} t
\end{align*}
where we have used the fact that $\int_{|x| \geq R} |\nabla u|^2 \, dx \lesssim R^{-1} $.
	 Since
	 \begin{equation*}
	 	\lim_{R \to \infty}V_{\ph_R}(0)=\lim_{R \to \infty}\int \ph_R(x)|u_0(x)|^2dx \leq \lim_{R \to \infty}\int_{|x|> \frac{R}{2}}|u_0(x)|^2dx=0,
	 \end{equation*}
	 we have  $V_{\ph_R}(0)=o_R(1)$. Hence
	 \begin{equation}\label{u<or1}
	 	\Vert u(t) \Vert_{L^2(|x|>R)}^2 \leq o_R(1)+\eps
	 \end{equation}
	 for any $\eps>0$, $R>0$ and all $t\in[0,T]$ with $T=\frac{\eps R}{C}$.
	 By \eqref{kdelta},\eqref{d2v<u}, \eqref{supdelta<infty} and \eqref{u<or1},  can be simplified as the following equation
	 \begin{equation*}
	 	\begin{aligned}
	 		\frac{d^2}{dt^2}V_{\ph_R}(t) \leq -8\delta+o_R(1) +C (o_R(1)+\eps)^{\frac{1}{2}}
	 	\end{aligned}
	 \end{equation*}
	 for all $t\in[0,T]$. Choosing $\eps>0$ small enough and $R>0$ large enough so that $$(o_R(1)+\eps)^{\frac{1}{2}}\leq 4\delta,$$
	 then we have
	 \begin{equation*}
	 	\frac{d^2}{dt^2}V_{\ph_R}(t) \leq -8\delta+4\delta=-4\delta
	 \end{equation*}
	 for all $t\in[0,T]$. It follows that
	 \begin{align*}
 		V_{\ph_R}(T)& \leq V_{\ph_R}(0)+V_{\ph_R}'(0)T-2\delta T^2 \notag \\
 		&\leq V_{\ph_R}(0)+V_{\ph_R}'(0)\frac{\eps R}{C}-2\delta (\frac{\eps R}{C})^2 \notag \\
 		&\leq o_R(1)R^2+o_R(1)\frac{\eps R^2}{C}-2\delta (\frac{\eps^2 R^2}{C^2}) \notag \\
 		&\leq 2 (o_R(1)-\frac{\delta \eps^2}{C^2})R^2 
        \leq-2 \frac{\delta \eps^2}{C^2} R^2
	 \end{align*}
	 where we have used $V_{\ph_R}(0)=o_R(1)R^2$ and $V_{\ph_R}'(0)=o_R(1)R$.
	 Taking $R>0$ large enough, we get $V_{\ph_R}(T)<0$ which is contradicted with the definition of $V_{\ph_R}(t)$ in \eqref{vt}. The proof is completed.
\end{proof}

\textbf{Statement}:
The authors have no competing interests to declare that are relevant to the content of this article.
The authors confirm that the data supporting the findings of this study are available within the article.

%
%
%
%

\end{document}